\documentclass[11pt]{amsart}
\usepackage{amsmath}
\usepackage{slashbox}
\usepackage{amsfonts,amssymb}
\usepackage{graphicx}
\usepackage{enumerate}
\usepackage{amsthm}
\usepackage[all]{xy}
\usepackage{extarrows}
\usepackage{color}
\usepackage{lmodern}
\usepackage{shuffle}
\usepackage{mathabx}
\usepackage{lmodern}
\usepackage{tikz}
\usepackage{hyperref}
\usepackage{scalefnt}
\usepackage{a4wide}

\newcommand {\bR}{\mathbb R}
\newcommand {\bN}{\mathbb N}
\newcommand {\bZ}{\mathbb Z}
\newcommand {\bC}{\mathbb C}

\newcommand {\bQ}{\mathbb Q}

\newcommand {\Id}{\operatorname{Id}}

\newcommand {\bk}{\mathbf{k}}

\newcommand {\be}{\mathbf{1}}

\newcommand{\cA}{\mathcal{A}}
\newcommand{\cB}{\mathcal{B}}
\newcommand{\cL}{\mathcal{L}}
\newcommand{\cH}{\mathcal{H}}

\newcommand{\cP}{\mathcal{P}}

\newcommand{\Hom}{\operatorname{Hom}}

\newcommand{\dpt}{\operatorname{dpt}}
\newcommand{\wt}{\operatorname{wt}}

\newcommand {\RE}{\operatorname{Re}}

\newcommand {\Li}{\operatorname{Li}}

\newcommand{\cS}{\mathcal S}

\newcommand{\fz}{\mathfrak{z}}

\newcommand{\ofz}{\overline{\mathfrak{z}}}

\newcommand {\z}{\zeta}

\newcommand{\sh}{\,\shuffle\,}
\newcommand{\shl}{\,\shuffle_\lambda\,}
\newcommand{\sho }{\,\shuffle_0\,}
\newcommand{\shm }{\,\shuffle_{-1}\,}

\newcommand{\dl}{\Delta_\lambda}
\newcommand{\odl}{\overline{\Delta}_\lambda}
\newcommand{\tdl}{\tilde{\Delta}_\lambda}
\newcommand{\cT}{\mathcal{T}}
\newcommand{\od}{\overline{\Delta}}
\newcommand{\MP}{\mathcal{MP}}

\long\def\ignore#1{}

\newtheorem{theorem}{Theorem}[section]
\newtheorem {lemma}[theorem]{Lemma}
\newtheorem {proposition}[theorem]{Proposition}

\newtheorem {corollary}[theorem]{Corollary}

\newtheorem {example}[theorem]{Example}
\newtheorem {remark}[theorem]{Remark}

\def\y{
\begin{tikzpicture}
\draw[fill=white] circle(2pt);
\end{tikzpicture}}

\def\d{
\begin{tikzpicture}
\filldraw[black] circle(2pt);
\end{tikzpicture}}

\def\ddydy{
\begin{tikzpicture}
\draw[thick=black] (0:1.5) arc (0:180:1.5cm);
\draw[fill=white] (-1.65,0) -- (1.65,0);
\filldraw[black] (-1.5,0) circle(2pt);
\filldraw[black] (45:1.5cm) circle(2pt);
\draw[fill=white] (90:1.5cm) circle(2pt);
\filldraw[black] (135:1.5cm) circle(2pt);
\draw[fill=white] (1.5,0) circle(2pt);
\end{tikzpicture}}

\def\Py{
\begin{tikzpicture}
\draw[thick=black] (0:1.5) arc (0:180:1.5cm);
\draw[fill=white] (-1.65,0) -- (1.65,0);
\filldraw[black] (-1.5,0) circle(2pt);
\filldraw[black] (45:1.5cm) circle(2pt);
\draw[fill=white] (90:1.5cm) circle(2pt);
\draw (90:1.5cm) circle(4pt);
\filldraw[black] (135:1.5cm) circle(2pt);
\draw[fill=white] (1.5,0) circle(2pt);
\end{tikzpicture}}

\def\PdyA{
\begin{tikzpicture}
\draw[thick=black] (0:1.5) arc (0:180:1.5cm);
\draw[fill=white] (-1.65,0) -- (1.65,0);
\filldraw[black] (-1.5,0) circle(2pt);
\draw[fill=white] (-1.5,0) -- (90:1.5cm);
\filldraw[black] (45:1.5cm) circle(2pt);
\draw[fill=white] (90:1.5cm) circle(2pt);
\filldraw[black] (135:1.5cm) circle(2pt);
\draw[fill=white] (1.5,0) circle(2pt);
\end{tikzpicture}}

\def\PdyAx{
\begin{tikzpicture}
\draw[thick=black] (0:1.5) arc (0:180:1.5cm);
\draw[fill=white] (-1.65,0) -- (1.65,0);
\filldraw[black] (-1.6,0) circle(2pt);
\filldraw[black] (-1.40,0) circle(2pt);
\draw[fill=white] (-1.5,0) -- (90:1.5cm);
\filldraw[black] (45:1.5cm) circle(2pt);
\draw[fill=white] (90:1.5cm) circle(2pt);
\filldraw[black] (135:1.5cm) circle(2pt);
\draw[fill=white] (1.5,0) circle(2pt);
\end{tikzpicture}}

\def\PdyB{
\begin{tikzpicture}
\draw[thick=black] (0:1.5) arc (0:180:1.5cm);
\draw[fill=white] (-1.65,0) -- (1.65,0);
\filldraw[black] (-1.5,0) circle(2pt);
\filldraw[black] (45:1.5cm) circle(2pt);
\draw[fill=white] (90:1.5cm) circle(2pt);
\filldraw[black] (135:1.5cm) circle(2pt);
\draw[fill=white] (135:1.5cm) -- (90:1.5cm);
\draw[fill=white] (1.5,0) circle(2pt);
\end{tikzpicture}}

\def\PdyBx{
\begin{tikzpicture}
\draw[thick=black] (0:1.5) arc (0:180:1.5cm);
\draw[fill=white] (-1.65,0) -- (1.65,0);
\filldraw[black] (-1.5,0) circle(2pt);
\filldraw[black] (45:1.5cm) circle(2pt);
\draw[fill=white] (90:1.5cm) circle(2pt);
\filldraw[black] (135:1.6cm) circle(2pt);
\draw[fill=white] (135:1.5cm) -- (90:1.5cm);
\filldraw[black] (135:1.40cm) circle(2pt); 
\draw[fill=white] (1.5,0) circle(2pt);
\end{tikzpicture}}

\def\PdyC{
\begin{tikzpicture}
\draw[thick=black] (0:1.5) arc (0:180:1.5cm);
\draw[fill=white] (-1.65,0) -- (1.65,0);
\filldraw[black] (-1.5,0) circle(2pt);
\filldraw[black] (45:1.5cm) circle(2pt);
\draw[fill=white] (45:1.5cm) -- (1.5,0);
\draw[fill=white] (90:1.5cm) circle(2pt);
\filldraw[black] (135:1.5cm) circle(2pt);
\draw[fill=white] (1.5,0) circle(2pt);
\end{tikzpicture}}

\def\PddyA{
\begin{tikzpicture}
\draw[thick=black] (0:1.5) arc (0:180:1.5cm);
\draw[fill=white] (-1.65,0) -- (1.65,0);
\filldraw[black] (-1.5,0) circle(2pt);
\filldraw[black] (45:1.5cm) circle(2pt);
\draw[fill=white] (90:1.5cm) circle(2pt);
\filldraw[black] (135:1.5cm) circle(2pt);
\draw[fill=white] (135:1.5cm) -- (90:1.5cm);
\draw[fill=white] (1.5,0) circle(2pt);
\draw[fill=white] (-1.5,0) -- (135:1.5cm);
\draw[fill=white] (-1.5,0) -- (90:1.5cm);
\end{tikzpicture}}

\def\PddyAxa{
\begin{tikzpicture}
\draw[thick=black] (0:1.5) arc (0:180:1.5cm);
\draw[fill=white] (-1.65,0) -- (1.65,0);
\filldraw[black] (-1.6,0) circle(2pt);
\filldraw[black] (-1.4,0) circle(2pt);
\filldraw[black] (45:1.5cm) circle(2pt);
\draw[fill=white] (90:1.5cm) circle(2pt);
\filldraw[black] (135:1.5cm) circle(2pt);
\draw[fill=white] (135:1.5cm) -- (90:1.5cm);
\draw[fill=white] (1.5,0) circle(2pt);
\draw[fill=white] (-1.5,0) -- (135:1.5cm);
\draw[fill=white] (-1.5,0) -- (90:1.5cm);
\end{tikzpicture}}

\def\PddyAxb{
\begin{tikzpicture}
\draw[thick=black] (0:1.5) arc (0:180:1.5cm);
\draw[fill=white] (-1.65,0) -- (1.65,0);
\filldraw[black] (-1.5,0) circle(2pt);
\filldraw[black] (45:1.5cm) circle(2pt);
\draw[fill=white] (90:1.5cm) circle(2pt);
\filldraw[black] (135:1.6cm) circle(2pt);
\draw[fill=white] (135:1.5cm) -- (90:1.5cm);
\filldraw[black] (135:1.4cm) circle(2pt);
\draw[fill=white] (1.5,0) circle(2pt);
\draw[fill=white] (-1.5,0) -- (135:1.5cm);
\draw[fill=white] (-1.5,0) -- (90:1.5cm);
\end{tikzpicture}}

\def\PddyAxx{
\begin{tikzpicture}
\draw[thick=black] (0:1.5) arc (0:180:1.5cm);
\draw[fill=white] (-1.65,0) -- (1.65,0);
\filldraw[black] (-1.6,0) circle(2pt);
\filldraw[black] (-1.4,0) circle(2pt);
\filldraw[black] (45:1.5cm) circle(2pt);
\draw[fill=white] (90:1.5cm) circle(2pt);
\filldraw[black] (135:1.6cm) circle(2pt);
\filldraw[black] (135:1.4cm) circle(2pt);
\draw[fill=white] (135:1.5cm) -- (90:1.5cm);
\draw[fill=white] (1.5,0) circle(2pt);
\draw[fill=white] (-1.5,0) -- (135:1.5cm);
\draw[fill=white] (-1.5,0) -- (90:1.5cm);
\end{tikzpicture}}

\def\PddyB{
\begin{tikzpicture}
\draw[thick=black] (0:1.5) arc (0:180:1.5cm);
\draw[fill=white] (-1.65,0) -- (1.65,0);
\filldraw[black] (-1.5,0) circle(2pt);
\draw[fill=white] (1.5,0) -- (135:1.5cm);
\filldraw[black] (45:1.5cm) circle(2pt);
\draw[fill=white] (45:1.5cm) -- (1.5,0);
\draw[fill=white] (90:1.5cm) circle(2pt);
\filldraw[black] (135:1.5cm) circle(2pt);
\draw[fill=white] (135:1.5cm) -- (45:1.5cm);
\draw[fill=white] (1.5,0) circle(2pt);
\end{tikzpicture}}

\def\PddyBx{
\begin{tikzpicture}
\draw[thick=black] (0:1.5) arc (0:180:1.5cm);
\draw[fill=white] (-1.65,0) -- (1.65,0);
\filldraw[black] (-1.5,0) circle(2pt);
\draw[fill=white] (1.5,0) -- (135:1.5cm);
\filldraw[black] (45:1.5cm) circle(2pt);
\draw[fill=white] (45:1.5cm) -- (1.5,0);
\draw[fill=white] (90:1.5cm) circle(2pt);
\filldraw[black] (135:1.6cm) circle(2pt);
\filldraw[black] (135:1.4cm) circle(2pt);
\draw[fill=white] (135:1.5cm) -- (45:1.5cm);
\draw[fill=white] (1.5,0) circle(2pt);
\end{tikzpicture}}

\def\PddyC{
\begin{tikzpicture}
\draw[thick=black] (0:1.5) arc (0:180:1.5cm);
\draw[fill=white] (-1.65,0) -- (1.65,0);
\filldraw[black] (-1.5,0) circle(2pt);
\draw[fill=white] (-1.5,0) -- (45:1.5cm);
\filldraw[black] (45:1.5cm) circle(2pt);
\draw[fill=white] (45:1.5cm) -- (1.5,0);
\draw[fill=white] (90:1.5cm) circle(2pt);
\filldraw[black] (135:1.5cm) circle(2pt);
\draw[fill=white] (1.5,0) circle(2pt);
\end{tikzpicture}}

\def\PddyCx{
\begin{tikzpicture}
\draw[thick=black] (0:1.5) arc (0:180:1.5cm);
\draw[fill=white] (-1.65,0) -- (1.65,0);
\filldraw[black] (-1.6,0) circle(2pt);
\filldraw[black] (-1.4,0) circle(2pt);
\draw[fill=white] (-1.5,0) -- (45:1.5cm);
\filldraw[black] (45:1.5cm) circle(2pt);
\draw[fill=white] (45:1.5cm) -- (1.5,0);
\draw[fill=white] (90:1.5cm) circle(2pt);
\filldraw[black] (135:1.5cm) circle(2pt);
\draw[fill=white] (1.5,0) circle(2pt);
\end{tikzpicture}}

\def\Pdddy{
\begin{tikzpicture}
\draw[thick=black] (0:1.5) arc (0:180:1.5cm);
\draw[fill=white] (-1.65,0) -- (1.65,0);
\filldraw[black] (-1.5,0) circle(2pt);
\filldraw[black] (45:1.5cm) circle(2pt);
\draw[fill=white] (45:1.5cm) -- (1.5,0);
\draw[fill=white] (90:1.5cm) circle(2pt);
\filldraw[black] (135:1.5cm) circle(2pt);
\draw[fill=white] (135:1.5cm) -- (45:1.5cm);
\draw[fill=white] (1.5,0) circle(2pt);
\draw[fill=white] (-1.5,0) -- (135:1.5cm);
\end{tikzpicture}}

\def\Pdddyxa{
\begin{tikzpicture}
\draw[thick=black] (0:1.5) arc (0:180:1.5cm);
\draw[fill=white] (-1.65,0) -- (1.65,0);
\filldraw[black] (-1.6,0) circle(2pt);
\filldraw[black] (-1.4,0) circle(2pt);
\filldraw[black] (45:1.5cm) circle(2pt);
\draw[fill=white] (45:1.5cm) -- (1.5,0);
\draw[fill=white] (90:1.5cm) circle(2pt);
\filldraw[black] (135:1.5cm) circle(2pt);
\draw[fill=white] (135:1.5cm) -- (45:1.5cm);
\draw[fill=white] (1.5,0) circle(2pt);
\draw[fill=white] (-1.5,0) -- (135:1.5cm);
\end{tikzpicture}}

\def\Pdddyxb{
\begin{tikzpicture}
\draw[thick=black] (0:1.5) arc (0:180:1.5cm);
\draw[fill=white] (-1.65,0) -- (1.65,0);
\filldraw[black] (-1.5,0) circle(2pt);
\filldraw[black] (45:1.5cm) circle(2pt);
\draw[fill=white] (45:1.5cm) -- (1.5,0);
\draw[fill=white] (90:1.5cm) circle(2pt);
\filldraw[black] (135:1.6cm) circle(2pt);
\filldraw[black] (135:1.4cm) circle(2pt);
\draw[fill=white] (135:1.5cm) -- (45:1.5cm);
\draw[fill=white] (1.5,0) circle(2pt);
\draw[fill=white] (-1.5,0) -- (135:1.5cm);
\end{tikzpicture}}

\def\Pdddyxx{
\begin{tikzpicture}
\draw[thick=black] (0:1.5) arc (0:180:1.5cm);
\draw[fill=white] (-1.65,0) -- (1.65,0);
\filldraw[black] (-1.6,0) circle(2pt);
\filldraw[black] (-1.4,0) circle(2pt);
\filldraw[black] (45:1.5cm) circle(2pt);
\draw[fill=white] (45:1.5cm) -- (1.5,0);
\draw[fill=white] (90:1.5cm) circle(2pt);
\filldraw[black] (135:1.6cm) circle(2pt);
\filldraw[black] (135:1.4cm) circle(2pt);
\draw[fill=white] (135:1.5cm) -- (45:1.5cm);
\draw[fill=white] (1.5,0) circle(2pt);
\draw[fill=white] (-1.5,0) -- (135:1.5cm);
\end{tikzpicture}}

\begin{document}

\title[Shuffle renormalization of ($q$)MZVs]{The Hopf algebra of ($q$)multiple polylogarithms\\ with non-positive arguments}

\author[K.~Ebrahimi-Fard]{Kurusch Ebrahimi-Fard}
\address{ICMAT,
		C/ Nicol\'as Cabrera~13-15, 28049 Madrid, Spain.
		On leave from Univ.~de Haute Alsace, Mulhouse, France}
         \email{kurusch@icmat.es, kurusch.ebrahimi-fard@uha.fr}         
         \urladdr{www.icmat.es/kurusch}

\author[D.~Manchon]{Dominique Manchon}
\address{Univ. Blaise Pascal, C.N.R.S.-UMR 6620, 3 place Vasar\'ely, CS 60026, 63178 Aubi\`ere, France}       
         \email{manchon@math.univ-bpclermont.fr}
         \urladdr{http://math.univ-bpclermont.fr/$\sim$manchon/}

\author[J.~Singer]{Johannes Singer}
\address{Department Mathematik, 
	Friedrich--Alexander--Universit\"at Erlangen--N\"urnberg, 
	Cauerstra\ss e 11,  
	91058 Erlangen, Germany}
\email{singer@math.fau.de}
\urladdr{http://math.fau.de/singer}
\date{\today}

\begin{abstract}
We consider multiple polylogarithms in a single variable at non-positive integers. Defining a connected graded Hopf algebra, we apply Connes' and Kreimer's algebraic Birkhoff decomposition to renormalize multiple polylogarithms at non-positive integer arguments, which satisfy the shuffle relation. The $q$-analogue of this result is as well presented, and compared to the classical case.
\end{abstract}

\maketitle

\tableofcontents


\section{Introduction}
\label{sect:intro}


Let $n, k_1, \ldots, k_n$ be positive integers. \emph{Multiple polylogarithms} (MPLs) in a single variable are defined by 
\begin{align*}
  \Li_{k_1,\ldots, k_n}(z) := \sum_{m_1>\cdots>m_n>0} \frac{z^{m_1}}{m_1^{k_1} \cdots m_n^{k_n}}
\end{align*}
when $z$ is a complex number. The function $\Li_{k_1,\ldots, k_n}(z)$ is of depth $\dpt(\mathbf{k}):=n \ge 1$ and weight $\wt(\mathbf k):=k_1+\cdots+k_n$, for $\mathbf{k}:=(k_1,\ldots,k_n)$. It is analytic in the open unit disk and, in the case $k_1 > 1$, continuous on the closed unit disk. In this case we observe the connection of MPLs and \emph{multiple zeta values} (MZVs)  
\begin{align}
\label{MZVs-sum}
 \zeta(k_1,\ldots,k_n):= \sum_{m_1>\cdots>m_n>0}\frac{1}{m_1^{k_1}\cdots m_n^{k_n}} =\Li_{k_1,\ldots,k_n}(1).
\end{align}

\noindent Equivalently we can define MPLs by induction on the weight $\wt(\mathbf k)$ as follows: 
\begin{align}\label{eq:diff1}
  z\frac{d}{dz}\Li_{k_1,\ldots,k_n}(z)=\Li_{k_1-1,k_2, \ldots, k_n}(z) \hspace{0.3cm} \text{if}\hspace{0.3cm} k_1>1, 
\end{align}
\begin{align}\label{eq:diff2}
  (1-z)\frac{d}{dz}\Li_{1,k_2,\ldots,k_n}(z)=\Li_{k_2,\ldots,k_n}(z) \hspace{0.3cm} \text{if}\hspace{0.3cm} n>1,
\end{align}
\begin{align}\label{eq:initial}
  \Li_{k_1,\ldots,k_n}(0)=0.
\end{align}

Therefore one can observe an integral formula for MPLs using iterated Chen integrals. Indeed, let $\varphi_1, \ldots, \varphi_p$ be complex-valued differential $1$-forms defined on a compact interval. Then we define inductively for real numbers $x$ and $y$
\begin{align*}
 \int_x^y \varphi_1\cdots \varphi_p:=\int_x^y\varphi_1(t)\int_x^t\varphi_2\cdots \varphi_p. 
\end{align*}
Now we set 
\begin{align*}
 \omega_{k_1,\ldots,k_n}:=\omega_0^{k_1-1}\omega_1 \cdots \omega_0^{k_n-1}\omega_1,
\end{align*}
where
\begin{align*}
 \omega_0(t):=\frac{dt}{t} \hspace{0.4cm} \text{and} \hspace{0.4cm} \omega_1(t):=\frac{dt}{1-t}. 
\end{align*}
Using the differential equations \eqref{eq:diff1}, \eqref{eq:diff2} and the initial conditions \eqref{eq:initial} we obtain
\begin{align}
\label{eq:integralrep}
  \Li_{k_1,\ldots,k_n}(z) =  \int_0^z \omega_{k_1,\ldots,k_n} 
\end{align}
using the convention $\Li_{\emptyset}(z)=1$. 
This representation gives rise to the well known shuffle products of MPLs and MZVs (see e.g. \cite{Waldschmidt02,Waldschmidt11}).\\

Recall that the $\mathbb Q$-vector space spanned by MZVs forms an algebra equipped with two products. The quasi-shuffle product is obtained when one multiplies series (\ref{MZVs-sum}) directly, which yields a linear combination of MZVs due to the product rule for sums. The aforementioned shuffle product between MZVs derives from integration by parts for iterated integrals. The resulting so-called double shuffle relations among MZVs arise from the interplay between these two products. An alternative characterization of MPLs can be given by the following formula
\begin{align}
\label{eq:Jcharact}
 \Li_{k_1,\ldots,k_n}(z) = J^{k_1}[yJ^{k_2}[y \cdots J^{k_n}[y]\cdots ]](z),
\end{align}
where $y(z):=\frac{z}{1-z}$ and $J[f](z):=\int_0^z \frac{f(t)}{t}\,dt$ (see Lemma \ref{lem:Jiteration}). Since $J$ is a Rota--Baxter operator of weight zero, iterations of the operator $J$ induce a product that coincides with the usual shuffle product for MPLs (see Lemma \ref{lem:jash}). For $|z|<1$ Equation \eqref{eq:Jcharact} is valid for all $k_1,\ldots,k_n\in \bZ$. The inverse of $J$ is given by $J^{-1}[f](t) = \delta[f](t):=t\frac{\partial f}{\partial t}(t)$ (Proposition \ref{prop:RBOJ}). 
We study the $\bQ$-vector space
\begin{align*}
 \MP:=\langle z\mapsto \Li_{-k_1,\ldots,-k_n}(z) \colon k_1,\ldots,k_n\in \bN_0, n\in \bN \rangle_{\bQ},
\end{align*}
which is indeed an algebra, where the product is induced by Equation \eqref{eq:Jcharact} (see Lemma \ref{lem:Jalgebra}).

The algebra $\mathcal{MP}$ admits also an interpretation for MZVs at non-positive integers. Indeed, let $k_1,\ldots,k_n \in \bN_0$. It is easily seen that $\Li_{-k_1,\ldots,-k_n}(z)$ is convergent for $|z|<1$ and divergent for $z=1$. Nevertheless we can perceive the product induced for $|z|<1$ as an analogue for the shuffle product of MZVs at non-positive integers. In order to make this connection more precise we have to establish a renormalization procedure. This permits us to extract explicit numbers for MZVs with non-positive arguments in a consistent way, such that they satisfy the shuffle product relations induced by the algebra structure of $\MP$.\\

We should keep in mind that a characterization of the shuffle product at non-positive integers -- in contrast to the quasi-shuffle product --  is a crucial point. Since the quasi-shuffle product is induced by the series representation of MZVs, the combinatorics is essentially the same as for positive arguments. On the other hand the shuffle product for positive indices is induced by the integral representation \eqref{eq:integralrep}. The combinatorics behind this product comes from shuffling of integration variables. It could be illustrated by the shuffling of two decks of cards, say a deck of red and blue cards, each consecutively numbered such that the internal numbering of red and blue cards is preserved. In this approach, however, it is not clear how to handle non-positive arguments, which corresponds to a non-positive number of cards.\\

Extracting finite numbers for MZVs at non-positive integers is accomplished by the process of renormalization, which involves two steps: 
\begin{enumerate}[(I)]

 \item introduction of a regularization scheme,

 \item applying a subtraction method.

\end{enumerate}

\noindent In step (I) we consider divergent MZVs $\zeta(-k_1,\ldots,-k_n)$ with $k_1,\ldots,k_n\in \bN_0$, and introduce a so-called regularization parameter $z$, which systematically deforms the divergent MZV in order to obtain a meromorphic function in $z$ with the only singularity in $z=0$. 
Step (II) involves a systematic procedure to eliminate singularities in terms of recursively defined subtractions. A rather natural way to achieve such eliminations is widely known as minimal subtraction scheme. The renormalization process is an integral part of perturbative quantum field theory (QFT). See e.g.~\cite{Collins}. The ``right" choice of the regularization scheme in QFT is essential in the light of constraints coming from physics. In our context those constraints are of mathematical nature: The deformation of divergent MZVs has to be established in such a way that the regularized MZVs coincide with the meromorphic continuation of (M)ZVs. The recursively defined subtractions in step (II) involve combinatorial structures, which are concisely captured by the Connes--Kreimer Hopf algebraic approach to renormalization \cite{Connes00, Connes01, Manchon08}.\\

One of the key points in our approach is based on providing an adequate Hopf algebra together with an algebra morphism from that Hopf algebra into the space $\MP$ (Theorem \ref{theo:HopfQuot}), which permits to define a consistent renormalization process. Regarding regularization schemes, we will use the fact, that MPLs may be considered as regularized classical MZVs. However, we also consider a specific $q$-analogue of MZVs \cite{Ohno12}, where the variable $q$ takes the role of a natural regulator.

\begin{remark} {\rm{Renormalization of MZVs at non-positive integers appeared already in \cite{Guo08} and \cite{Manchon10}. The authors applied regularization schemes together with well-chosen subtraction methods suitable for preserving the quasi-shuffle product for renormalized MZVs -- at non-positive arguments. Using Ecalle's Mould calculus, Bouillot proposes in his work \cite{Bouillot13} a unifying picture of MZVs at non-positive arguments respecting the quasi-shuffle product. The common point of our approach with those presented in the aforementioned references is the use of the Connes--Kreimer Hopf algebraic approach to renormalization, and the corresponding algebraic Birkhoff decomposition, which encodes the subtraction procedure for singularities. However, we should emphasize that in our work it is the shuffle product, in a naturally extended sense, which is satisfied by renormalized MZVs at non-positive arguments.
We remark that Zhao's approach in \cite{Zhao08} is rather different from the point of view we present in this paper, since it extended the approach of \cite{Guo08} to a particular $q$-analog of MZVs preserving the quasi-shuffle product, while not attributing any regularization properties to the $q$-parameter itself.}}
\end{remark}

In \cite{Castillo13b} the authors indicated that the $q$-parameter appearing in a specific $q$-analogue of MZVs \cite{Ohno12} (see Equation \eqref{eq:qMPLs} below) may be considered as a regularization parameter for MZVs at non-positive arguments. The approach presented in our paper can be consistently extended to this $q$-analogue of MZVs ($q$MZVs). Indeed, we would like to demonstrate that under the $q$-parameter regularization \cite{Castillo13a,Castillo13b,Ohno12} a proper renormalization of MZVs can be defined. The \emph{$q$-multiple polylogarithm} ($q$MPL) in one variable is defined as 
\begin{align}
\label{qMPL}
	\Li^q_{k_1,\ldots,k_n}(z):=\sum_{m_1>\cdots >m_n>0}\frac{z^{m_1}}{[m_1]_q^{k_1}\cdots [m_n]_q^{k_n}}
\end{align}
with $[m]_q:=\frac{1-q^m}{1-q}$. It turns out that for $|z|<1$ the series in (\ref{qMPL}) is convergent for $k_1,\ldots,k_n\in \bZ$, and especially for $|q|<1$ we obtain the formal power series 
\begin{align}
\label{eq:qMPLs}
	\fz_q(k_1,\ldots,k_n):= \Li^q_{k_1,\ldots,k_n}(q) 
					= \sum_{m_1>\cdots>m_n>0} \frac{q^{m_1}}{[m_1]_q^{k_1}\cdots [m_n]_q^{k_n}} \in \bZ[[q]].
\end{align}
These $q$MZVs were introduced by Ohno, Okuda and Zudilin in \cite{Ohno12}, and further studied in \cite{Castillo13a,Castillo13b}, see also \cite{Singer15,Singer14,Zhao14}. For $k_1 > 1$ and $k_2,\ldots,k_n\geq 1$ we see that 
\begin{align}\label{eq:lim}
	\lim_{q\nearrow 1}\fz_q(k_1,\ldots,k_n) = \z(k_1,\ldots,k_n), 
\end{align}
where $q \nearrow 1$ means $q\to 1$ inside an angular sector $-\frac{\pi}{2} +\varepsilon \leq \operatorname{Arg}(1-q) \leq  \frac{\pi}{2} - \varepsilon$ with $\varepsilon>0$ sufficiently small. It will be convenient for technical reasons to consider the \emph{modified $q$MZVs} introduced in \cite{Ohno12} 
\begin{align}\label{eq:modify}
	\ofz_q(k_1,\ldots,k_n) := (1-q)^{-(k_1+\cdots+k_n)}\fz(k_1,\ldots,k_n).
\end{align}
They are used to establish a Hopf algebra structure on the space of modified $q$MPLs. The modification has to be reversed after renormalization, in order to relate the renormalized $q$MZVs via \eqref{eq:lim} to renormalized MZVs. 

\medskip 

The paper is organized as follows. In Section \ref{sect:mero} we recall the basic results on the meromorphic continuation of MZVs. Section \ref{sect:algebra} contains the main result, i.e., the detailed construction of a Hopf algebra for MPLs at non-positive integers. A generalization of this finding to the $q$-analogue of MZVs defined by Ohno, Okuda and Zudilin is presented as well. In Section \ref{sect:Renorm} we recall the Hopf algebra approach to perturbative renormalization by Connes and Kreimer, and apply one of its main theorems to the renormalization of MPLs at non-positive integer arguments. The $q$-analogue of this result is as well presented, and compared to the classical case.
  
\medskip

\noindent {\bf{Acknowledgement}}: The first author is supported by a Ram\'on y Cajal research grant from the Spanish government. The second and third authors gratefully acknowledge support by ICMAT and the Severo Ochoa Excellence Program. The second author is supported by Agence Nationale de la Recherche (projet CARMA).\\
The authors gratefully acknowledge that most of the discussions and work was carried out during ICMAT Fall School Multiple Zeta Values, Multiple Polylogarithms and Quantum Field Theory (Oct. 7-11, 2013, ICMAT, Madrid), Clay Mathematics Institute Summer School 2014 (June 30 - July 25, 2014, Madrid) and ICMAT Research Trimester on Multiple Zeta Values, Multiple Polylogarithms, and Quantum Field Theory (Sept. 15 - Dec. 19, 2014, ICMAT, Madrid). 
%

\section{Meromorphic Continuation of MZVs}
\label{sect:mero}

In this section we review some well-known facts about the meromorphic continuation of MZVs. For $n \in \bN$ we consider the function 
\begin{align}
\label{eq:MRiemann}
 \z_n \colon \bC^n \to \bC, 
 \hspace{1cm} 
 \z_n(s_1,\ldots,s_n):= \sum_{m_1 > \cdots > m_n >0}\frac{1}{m_1^{s_1}\cdots m_n^{s_n}}.  
\end{align}

\begin{proposition}[\cite{Krattenthaler07}]
The domain of absolute convergence of the function \eqref{eq:MRiemann} is given by
\begin{align*}
 \left\{(s_1,\ldots,s_n)\in \bC^n\colon \sum_{j=1}^k\RE(s_j)>k, k=1,\ldots,n \right\}.
\end{align*}
In this domain $\zeta_n$ defines an analytic function in $n$ variables. 
\end{proposition}

\begin{theorem}[\cite{Akiyama01a,Akiyama01b,Manchon10}]
\label{theo:meroz}
The function $\z_n(s_1,\ldots,s_n)$ admits a meromorphic extension to $\bC^n$. The subvariety $\cS_n$ of singularities is given by
\begin{align*}
 \cS_n= \left\{(s_1,\ldots,s_n)\in \bC^n\colon s_1=1; s_1+s_2=2,1,0,-2,-4,\ldots; \sum_{i=1}^js_i\in \bZ_{\leq j} ~(j=3,4,\ldots,n)\right\}.
\end{align*}
\end{theorem}

\noindent In the subsequent sections $\z_n$ always denotes the meromorphic continuation of MZVs. 

\begin{remark}\label{rem:mero}
{\rm{In this paper we discuss $\z_n$ restricted to the set $(\bZ_{\leq0})^n$. The Bernoulli numbers are defined by the following generating series 
 \begin{align*}
  \frac{te^t}{e^t-1} = \sum_{m\geq0} \frac{B_m}{m!}t^m.
 \end{align*}
 The first few values are $B_0=1, B_1=\frac{1}{2}, B_2=\frac{1}{6},B_3=0, B_4=-\frac{1}{30}, B_5=0,$ etc., especially $B_{2l+1}=0$ for $l\in \bN$. 
 Therefore we have the following cases for $\z_n$ restricted to non-positive arguments: 
 
 \allowdisplaybreaks{
 \begin{itemize}
  \item Case $n=1$:  For $l \in \bN_0$ we have the well known formula 
  \begin{align*}
   \z_1(-l) = -\frac{B_{l+1}}{l+1}.
  \end{align*}
  
  \item Case $n=2$: In the light of Theorem \ref{theo:meroz} we assume the sum $k_1+k_2$ to be odd. Therefore we obtain from \cite{Akiyama01a} that
  \begin{align*}
   \z_2(-k_1,-k_2) = \frac{1}{2} \left(1+\delta_0(k_2) \right) \frac{B_{k_1+k_2+1}}{k_1+k_2+1}.
  \end{align*}
  
  \item Case $n\geq 3$: From Theorem \ref{theo:meroz} we deduce that
  \begin{align*}
   (\bZ_{\leq 0})^n \subseteq \cS_n.  
  \end{align*}
  Therefore we obtain no information from the meromorphic continuation.   
 \end{itemize}
 }}}
\end{remark}


\section{Algebraic Framework}
\label{sect:algebra}

We briefly introduce Rota--Baxter algebras, since they conveniently relate to shuffle-type products on word algebras. Two such shuffle products are presented, which encode products of MPLs and $q$MPLs at integer arguments. The main result of this section is the construction of a graded connected commutative and cocommutative shuffle Hopf algebra for ($q$)MPLs at non-positive integer arguments.


\subsection{Rota--Baxter Algebra and multiple zeta values}
\label{ssect:RBA}

Let $k$ be a ring, $\lambda \in k$ and $\cA$ a $k$-algebra. A \emph{Rota--Baxter operator (RBO) of weight $\lambda$ on $\cA$ over $k$} is a $k$-module endomorphism $L$ of $\cA$ such that 
\begin{align*}
	L(x)L(y) = L(xL(y)) + L(L(x)y) + \lambda L(xy) 
\end{align*}
for any $x,y \in \cA$. A \emph{Rota--Baxter $k$-algebra (RBA) of weight $\lambda$} is a pair $(\cA,L)$ with a $k$-algebra $\cA$ and a Rota--Baxter operator $L$ of weight $\lambda$ on $\cA$ over $k$. On the algebra of continuous functions $C(\bR)$ the integration operator 
\begin{align*}
	R\colon C(\bR) \to C(\bR), \hspace{0.5cm} R[f](z):=\int_{0}^zf(x)\,dx
\end{align*}
is a RBO of weight zero, which is an immediate consequence of the integration by parts formula. We consider the $\bC$-algebra of power series 
\begin{align*}
	\cP_{\geq 1}:=\bigg\{ f(t):=\sum_{k\geq 1}a_k t^k\colon R_f\geq 1 \bigg\}\subseteq t\bC[[t]]
\end{align*}
without a term of degree zero in $t$, and radius of convergence, $R_f$, of at least $1$. We define the operator 
\begin{align*}
	J\colon \cP_{\geq 1} \to\cP_{\geq 1}, \hspace{0.5cm} J[f](t):=\int_{0}^t f(z)\frac{dz}{z}.
\end{align*}
Further the \emph{Euler derivation} $\delta$ is given by 
\begin{align*}
	\delta\colon \cP_{\geq 1} \to \cP_{\geq 1}, \hspace{0.5cm} \delta[f](t):=t\frac{\partial f}{\partial t}(t). 
\end{align*}
 
\begin{proposition}\label{prop:RBOJ} ~
\begin{enumerate}[(i)]
	\item The pair $(\cP_{\geq 1},J)$ is a RBA of weight $\lambda=0$.
   	\item The operator $\delta$ is a derivation, i.e., $\delta[fg] = \delta[f]g + f\delta[g]$, for any $f,g\in \cP_{\geq 1}$.
   	\item The operators $J$ and $\delta$ are mutually inverse, i.e., $J\circ \delta = \delta \circ J = \Id$.
 \end{enumerate}
 \end{proposition}

\begin{proof}
Statement (i) follows form integration by part. The second claim is straightforward to show. Finally, item (iii) is an immediate consequence of the fundamental theorem of calculus together with the fact that $f(0)=0$ for any $f\in \cP_{\geq 1}$. 
\end{proof}

\begin{lemma}
\label{lem:Jiteration}
Let $k_1,\ldots,k_n$ be integers. Then $\Li_{k_1,\ldots,k_n}(t) \in \cP_{\geq 1}$, explicitly 
\begin{align*}
  \Li_{k_1,\ldots,k_n}(t)=J^{k_1}[y J^{k_2}[y \cdots J^{k_n}[y]\cdots]](t),
\end{align*}
where $y(t):=\frac{t}{1- t}\in \cP_{\geq 1}$. 
\end{lemma}

\begin{proof}
Using the fact that $J^{-1}=\delta$ we prove the claim for $\bk:=(k_1,\ldots,k_n)\in \bZ^n$ by induction on its depth, $\dpt(\bk)=n$. For $\dpt(\bk)=1$ we easily compute 
\begin{align*}
 J^k[y](t)& = \left.\begin{cases}
             \sum_{m\geq 1}\frac{t^m}{m^k}, & \text{for~} k\geq 0 \\
             \sum_{m\geq 1}m^{|k|}{t^m} = \sum_{m\geq 1}\frac{t^m}{m^k},  &  \text{for~} k< 0
            \end{cases} \right\}
           = \Li_k(t) 
\end{align*}
for any $k \in \bZ$. In the inductive step we get 
\allowdisplaybreaks{
\begin{align*}
    J^{k_1}[y J^{k_2}[y \cdots J^{k_n}[y]\cdots]](t) 
 & = J^{k_1}\left[\sum_{m>0}t^m \sum_{m_2 > \cdots > m_n >0}\frac{t^{m_2}}{m_2^{k_2}\cdots m_n^{k_n}} \right] \\
 & = J^{k_1}\left[ \sum_{m_1 >m_2 > \cdots >m_n >0}\frac{t^{m_1}}{m_2^{k_2}\cdots m_n^{k_n}} \right] \\
 & = \sum_{m_1 >m_2 > \cdots >m_n >0}\frac{t^{m_1}}{m_1^{k_1}m_2^{k_2}\cdots m_n^{k_n}} \\
 & = \Li_{k_1,\ldots,k_n}(t)
\end{align*}}
using the induction hypothesis.  
\end{proof}

This lemma gives rise to the following algebraic formalism. Let $X_0:=\{j,d,y\}$, and $W_0$ denotes the set of words on the alphabet $X_0$, subject to the rule $jd=dj=\be$, where $\be$ denotes the empty word. Therefore any word $w\in W_0$ can be uniquely written in the canonical form
\begin{align*}
 w=j^{k_1}yj^{k_2}y\cdots j^{k_{n-1}}y j^{k_n}
\end{align*}
for $k_1,\ldots,k_n\in \bZ$ using the notation $j^{-1}=d$ and $j^0=\be$. The length of the word $w$ above is $|w|=k_1+\cdots+k_n+n-1$.
Further, $\cA_0$ denotes the vector space $\cA_0:=\langle W_0 \rangle_\bQ$  spanned by the words in $W_0$.
Next we define the product $\sho  \colon \cA_0\otimes \cA_0 \to \cA_0$  by $\be\sho w := w \sho \be := w$ for any word $w\in W_0$, and recursively with respect to the sum of the length of two words in $W_0$:
\allowdisplaybreaks{
\begin{enumerate}[(i)]
 \item $yu\sho   v := u \sho  yv := y(u\sho  v)$,
 \item $ju \sho  jv := j(u\sho  jv) + j(ju\sho  v)$,
 \item $du\sho  dv := d(u \sho  dv) - u \sho  d^2v $,
 \item $du\sho  jv := d(u\sho  jv)-u \sho  v$,
 \item $ju\sho dv := d(ju\sho v)-u\sho v$.
\end{enumerate}}
\begin{remark}{\rm{

  \textbullet~ Note that (iv) can be deduced from (iii) by replacing $v$ by $j^2v$. \\
    \textbullet~ (iii) does not really define $du\sho  dv$ by induction on the sum of lengths of two words, because $|du|+|dv|=|u|+|d^2v|$. Using (i) and writing $u'=du=d^kyw$ for some $k\ge 1$, we can however get a recursive definition by iterating (iii) as follows:
\begin{align*}
  d^kyw\sho  dv =&~d\big(d^{k-1}yw\sho  dv-d^{k-2}yw\sho  d^2v+\cdots\\
  & +(-1)^{k-1}yw\sho  d^kv\big)+(-1)^ky(w\sho  d^{k+1}v).
\end{align*}
 }}
\end{remark}

\begin{lemma}\label{lem:idealT}
The $\bQ$-vector space 
\begin{align*}
 \cT:=\langle j^{k_1}yj^{k_2}y\cdots j^{k_{n-1}}y j^{k_n}\in W_0 \colon k_n\neq 0, n\in \bN \rangle_{\bQ}
\end{align*}
is a two sided ideal of $(\cA_0,\sho)$.
\end{lemma}
\begin{proof}
Let $a\in \{j,d\}$ and $u:=u'a\in W_0$ and $v\in W_0$. We prove $u\sho v \in \cT$ by induction on $r:=|u|+|v|$. The base cases are true because we observe for $r=1$
that $d\sho  \be = d, j\sho \be = j$ and for $r=2$ that 
\begin{align*}
  d\sho y = yd, \hskip 8mm j\sho y = yj,\\ 
  d\sho d = 0,\hskip 4mm j\sho j = 2j^2,\hskip 4mm j\sho d=d\sho j=0.
\end{align*}
For the inductive step we have several cases: \\
\noindent\textbullet~ 1st case: $u=y\tilde{u}a$ or $v=y\tilde{v}$. This is an immediate consequence of (i) and the induction hypothesis. \\

\noindent\textbullet~ 2nd case: $u=j\tilde{u}a$ and $v=j\tilde{v}$. We observe using (ii) and the induction hypothesis that 
\begin{align*}
 j\tilde{u}a \sho j\tilde{v} = j(\tilde{u}a \sho j\tilde{v} + j\tilde{u}a \sho\tilde{v}) \in \cT. 
\end{align*}

\noindent\textbullet~ 3rd case: $u=d\tilde{u}a$ and $v=j\tilde{v}$. We observe using (iv) and the induction hypothesis that 
\begin{align*}
 d\tilde{u}a \sho j\tilde{v} = d(\tilde{u}a \sho j\tilde{v}) - \tilde{u}a \sho \tilde{v} \in \cT. 
\end{align*}

\noindent\textbullet~ 4th case: $u=j\tilde{u}a$ and $v=d\tilde{v}$. We observe using (v) and the induction hypothesis that 
\begin{align*}
 j\tilde{u}a \sho d\tilde{v} = d(j\tilde{u}a \sho \tilde{v}) - \tilde{u}a \sho \tilde{v} \in \cT. 
\end{align*}

\noindent\textbullet~ 5th case: $u=d\tilde{u}a$ and $v=d\tilde{v}$. We observe using (iii) that 
\begin{align*}
 d\tilde{u}a \sho d\tilde{v} = d(\tilde{u}a \sho d\tilde{v}) - \tilde{u}a \sho d^2\tilde{v}. 
\end{align*}
By induction hypothesis the first term is an element of $\cT$. If $\tilde{u}a$ is not a word consisting purely of $d$ we apply rule (iii) until we hit a letter not equal to $d$ and we are in one of the above cases. Therefore we only consider the case, where $\tilde{u}a$ is a word consisting purely of $d$, i.e., $\tilde{u}a=d^n$ for $n\in \bN$. Now we prove $d^n\sho d^mw=0$ for any $w\in W_0$ and $m\in \bN$. For $n=1$ we have $d\sho d^m w= d(\be \sho d^m w)- \be \sho d^{m+1}w = 0$. Therefore we obtain by induction hypothesis that 
\begin{align*}
 d^{n+1} \sho d^m w = d(d^{n}\sho d^{m}w) - d^n\sho d^{m+1}w =0. 
\end{align*}
All in all we have shown that $\sho(\cT \otimes \cA_0 )\subseteq \cT $.

Since $\sho$ is not commutative we also have to prove $v\sho u \in \cT$ by induction on $r:=|v|+|u|$. The base cases are true. The first four cases are completely analogous to the first four cases above. We only discuss the following case: $v=d\tilde{v}$ and $u=d\tilde{u}a$. \\
We observe using (iii) that 
\begin{align*}
  d\tilde{v} \sho d\tilde{u}a = d(\tilde{v}  \sho d \tilde{u}a ) - \tilde{v} \sho d^2\tilde{u}a. 
\end{align*}
By induction hypothesis the first term is an element of $\cT$. If $\tilde{v}$ starts with $j$ or $y$ we are in one of the above cases. Only the case $\tilde{v}=d^n$ for $n\in \bN$ has to be considered. By the same induction as in the previous case we obtain that the last term is zero. This proves $\sho(\cA_0 \otimes \cT )\subseteq \cT$. The proof is now complete. 
\end{proof}

Let $Y_0:=\{\be\}\cup W_0y$ be the set of \textsl{admissible words}, i.e., words which do not end up with a $j$ or a $d$. It is easily seen that $\cA'_0:=\langle Y_0 \rangle_\bQ$ is a subalgebra of $(\cA_0,\sho)$ isomorphic to $\cA_0/\cT$. A priori, the product $\sho$ on $\cA_0$ is neither commutative nor associative. Now let $\cL$ (resp. $\cL'$) be the ideal of $\cA_0$ (resp. $\cA'_0$) generated by 
$$\{j^k\big(d(u\sho v)-du\sho v-u\sho dv\big),\,u,v\in W_0y, \,k\in\bZ\}.$$
Let $\cB_0$ (resp. $\cB'_0$) be the quotient algebra $\cA_0/\cL$ (resp. $\cA'_0/\cL'$). We obviously have the isomorphism:
$$\cB'_0\sim \cA_0/(\cT+\cL).$$

\begin{proposition}\label{lem:Jalgebra}
The pair $(\cB_0,\sho)$ is a commutative, associative and unital algebra.
\end{proposition}

\begin{proof}
We first prove commutativity $u'\sho v'-v'\sho u'\in \cL$ by induction on $r=|u'|+|v'|$. The cases $r=0$ and $r=1$ are immediate. Several cases must be considered:

\noindent\textbullet~1st case: $u'=yu$ or $v'=yv$. The induction hypothesis immediately applies, using (i).\\

\noindent\textbullet~2nd case: $u'=ju$ and $v'=jv$. Then we have by induction hypothesis:
\begin{align*}
ju\sho jv-jv\sho ju=j(ju\sho v+u\sho jv-jv\sho u-v\sho ju)\in\cL.
\end{align*}

\noindent\textbullet~3rd case: $u'=du$ and $v'=jv$ or vice-versa. We have then:
\begin{align*}
du\sho jv-jv\sho du&=d(u\sho jv)-u\sho v-d(jv\sho u)+v\sho u\\
&=d(u\sho jv-jv\sho u)-(u\sho v-v\sho u),
\end{align*}
which belongs to $\cL$ by induction hypothesis.\\

\noindent\textbullet~4th case: $u'=du$ and $v'=dv$. Then $dv\sho du-d(dv\sho u)+d^2v\sho u\in\cL$, hence:
\begin{align*}
du\sho dv-dv\sho du&=d(u\sho dv)-u\sho d^2v-d(dv\sho u)+d^2v\sho u \mod \cL\\
&= d(u\sho dv-dv\sho u)-(u\sho d^2v-d^2v\sho u)\mod \cL.
\end{align*}
The first term belongs to $\cL$ by induction hypothesis. We further suppose that $u'$ is written $d^kyw$ for some $k\ge 1$ and $w\in Y_0$. Iterating the process using (iii) we finally get $du\sho dv-dv\sho du=(-1)^k(yw\sho d^{k+1}v-d^{k+1}v\sho yw)\mod\cL$. We are then back to the first case.

\vskip 3mm

Associativity follows by showing $u'\sho (v'\sho  w') = (u'\sho  v')\sho  w'$ via induction on the sum $|u'|+|v'|+|w'|$. If one of the words is the empty one nothing is to show. Now let $u'=au$, $v'=bv$ and $w'=cw$ with $a,b,c\in \{d,j,y\}$. \\

\noindent\textbullet~1st case: one of the letters is $y$, for example $u'=yu$. Using the induction hypothesis, we obtain
\allowdisplaybreaks{
\begin{align*}
 (yu\sho  v')\sho  w' 
 = (y(u\sho  v'))\sho  w' &= y((u\sho  v')\sho  w') \\
 &= y(u\sho  (v'\sho  w'))\mod \cL\\
 &= yu \sho  (v'\sho  w')\mod\cL.
\end{align*}}
Note that the other cases $v'=yv$ or $w'=yw$ are similar, and the arguments are completely analogous.\\

\noindent\textbullet~2nd case: $a=b=c=j$. On the one hand we have 
\allowdisplaybreaks{
\begin{align*}
 (ju\sho  jv)\sho  jw  
 =&~ j((u\sho  jv)\sho  jw) + j((ju\sho  v)\sho  jw) \\
  & + j(j(u\sho  jv)\sho w) + j (j(ju \sho  v)\sho w)\\
 =&~ j((u\sho  jv)\sho  jw) + j((ju\sho  v)\sho  jw) + j((ju\sho  jv)\sho  w), 
\end{align*}}
on the other hand 
\allowdisplaybreaks{
\begin{align*}
 ju\sho  (jv\sho  jw)
 =&~ j(u\sho  j(v\sho  jw)) + j(u\sho  j(jv\sho  w)) \\
  & + j(ju\sho  (v\sho  jw)) + j (ju \sho  (jv\sho  w))\\
 =&~ j(u\sho  (jv\sho  jw)) + j(ju\sho  (v\sho  jw)) + j(ju\sho  (jv\sho  w)).
\end{align*}}
Hence $(ju\sho  jv)\sho  jw =ju\sho  (jv\sho  jw)\mod \cL$.\\

\noindent\textbullet~3rd case: two $j$'s and one $d$. On the one hand we have 
\allowdisplaybreaks{
\begin{align*}
 (ju \sho  jv) \sho  dw =&~ d(j(u\sho  jv)\sho  w) - (u\sho  jv)\sho  w \\
 &+ d(j(ju\sho  v)\sho  w) - (ju \sho  v)\sho  w \\
 =&~d((ju\sho  jv)\sho  w) - (u\sho  jv)\sho  w - (ju\sho  v)\sho  w,
\end{align*}}
on the other hand 
\begin{align*}
 ju \sho  (jv \sho  dw) &= ju\sho  d(jv\sho  w)-ju\sho (v\sho  w)\\
 &= d(ju\sho (jv\sho w)) - u\sho (jv \sho  w) - ju \sho (v\sho  w).
\end{align*}
then $(ju \sho  jv) \sho  dw=ju \sho  (jv \sho  dw)\mod \cL$.\\

\noindent\textbullet~4th case: two $d$'s and one $j$. We have to prove 
\begin{align*}
 (du \sho  dv) \sho  jw =  du \sho  (dv \sho  jw)\mod\cL. 
\end{align*}
It suffices to show that 
\begin{align*}
 (d^kyu \sho  dv) \sho  jw =  d^kyu \sho  (dv \sho  jw)\mod\cL
\end{align*}
with $u\in W_0$ and $k\in \bN$. Using (iii) we observe 
\allowdisplaybreaks{
\begin{align*}
 & (d^kyu \sho  dv) \sho  jw =d(d^{k-1}yu\sho dv)\sho jw-(d^{k-1}yu\sho d^2v)\sho jw\\
 &=d\big((d^{k-1}yu\sho dv)\sho  jw\big) - (d^{k-1}yu\sho dv)\sho  w-(d^{k-1}yu\sho  d^2v)\sho  jw \\
 &=d\big(d^{k-1}yu\sho (dv\sho  jw)\big) - d^{k-1}yu\sho (dv\sho  w)	-(d^{k-1}yu\sho  d^2v)\sho  jw \mod\cL\\
 &=d^kyu\sho (dv\sho  jw) + d^{k-1} yu\sho  (d^2v\sho  jw) - (d^{k-1} yu\sho  d^2v)\sho jw \mod\cL . 
\end{align*}}
For the difference of the two terms in the previous line to belong to $\cL$, it suffices to prove 
$$ 
	d^{k-1} yu\sho  (d^2v\sho  jw) =(d^{k-1} yu\sho  d^2v)\sho  jw\mod\cL.
$$
Applying the above procedure iteratively this could be reduced to  
\begin{align*}
 yu\sho  (d^kv\sho  jw)  = (yu\sho  d^kv)\sho  jw \mod\cL, 
\end{align*}
which is true by using (i) and the induction hypothesis.\\

\noindent\textbullet~5th case: $a=b=c=d$. We have to prove 
\begin{align*}
 (du \sho  dv) \sho  dw =  du \sho  (dv \sho  dw)\mod\cL.
\end{align*}
It suffices to show that 
\begin{align*}
 (d^kyu \sho  dv) \sho  dw =  d^kyu \sho  (dv \sho  dw)\mod\cL,
\end{align*}
with $u\in W_0$ and $k\in \bN$. Using (iii) we observe 
\allowdisplaybreaks{
\begin{align*}
 & (d^kyu\sho  dv) \sho  dw = d(d^{k-1}yu\sho dv)\sho dw-(d^{k-1}yu\sho d^2v)\sho dw\\
 & = d\big((d^{k-1}yu\sho  dv)\sho  dw\big) - (d^{k-1}yu\sho  dv)\sho  d^2w -(d^{k-1}yu\sho d^2v)\sho  dw \\
 & = d^kyu\sho (dv\sho  dw) + d^{k-1}yu\sho (d^2v\sho  dw)+ d^{k-1}yu\sho (dv\sho  d^2w) \\
 & ~-(d^{k-1}yu\sho d^2v)\sho  dw -  (d^{k-1}yu\sho dv)\sho  d^2w\mod\cL.
\end{align*}}
Iteratively applying this procedure leads -- as in the 4th case -- to the claim using (i) and the induction hypothesis. Proposition \ref{lem:Jalgebra} is thus proven.
\end{proof}
\noindent Now we define the map $\z_t^\shuffle\colon \cB'_0 \to \bQ[[t]]$ by  $\z_t^\shuffle(\be):=1$, and for any $k_1,\ldots,k_n\in \bZ$,
\begin{align*}
	j^{k_1}y\cdots j^{k_n}y \mapsto \z_t^\shuffle(j^{k_1}y\cdots j^{k_n}y) := \Li_{k_1,\ldots,k_n}(t).
\end{align*}

\begin{lemma}\label{lem:characterJ}
The map $\z_t^\shuffle$ is multiplicative, i.e., is an algebra morphism. 
\end{lemma}

\begin{proof}
From Proposition \ref{prop:RBOJ} (ii) and (iii) we obtain for any $f,g\in \cP_{\geq 1}$
\begin{align}
\label{eq:mixed}
	\delta[J[f]g] = J[f]\delta[g] + fg.
\end{align}
Therefore the definition of $\sho $ and Proposition \ref{prop:RBOJ} (i), (ii), (iii) and \eqref{eq:mixed} imply that 
\begin{align*}
	\z_t^{\sh} \colon \cB'_0\to \bQ[[t]], 
	\hspace{0.5cm} j^{k_1}y \cdots j^{k_n}y \mapsto J^{k_1}[y \cdots J^{k_n}[y]\cdots](t)
\end{align*}
with $k_1,\ldots, k_n\in \bZ$ is an algebra morphism. 
\end{proof}

Next we show that if we restrict the shuffle product $\sho $ to admissible words corresponding to positive arguments we obtain the ordinary shuffle product. Let $\mathcal{C}:=\bQ\be\oplus j \bQ \langle j,y \rangle y$ and $\mathcal{D}:=\bQ\be\oplus x_0\bQ\langle x_0,x_1\rangle x_1$.
\begin{lemma}\label{lem:jash}
 The algebras $(\mathcal{C},\sho )$ and $(\mathcal{D},\sh)$ are isomorphic, where $\sh$ denotes the ordinary shuffle product. 
\end{lemma}
\begin{proof}
It is easily seen that $ \Phi \colon  (\mathcal{D},\sh) \to (\mathcal{C},\sho )$ given by $\be \mapsto \be$ and
\begin{align*}
  x_0^{k_1-1}x_1 x_0^{k_2-1}x_1 \cdots x_0^{k_n-1}x_1 \mapsto j^{k_1}y j^{k_2}y \cdots j^{k_n}y
\end{align*}
is an algebra morphism, for $k_1,\ldots,k_n\in \bN$, with $k_1>1$, $n\in \bN$. Since $\Phi$ is bijective the proof is complete. 
\end{proof}


\subsection{$q$-multiple zeta values}
\label{ssect:qRBA}

For a formal power series $f\in \bQ[[t]]$ we define the $q$-dilation operator as $E_q[f](t):=f(qt).$ Let $\cA:=t\bQ[[t,q]]$ be the space of formal power series in the variables $t$ and $q$, without a term of degree zero in $t$. We can interpret $\cA$ as the $\bQ[[q]]$-algebra $t\bQ[[t]]$. Then the $\bQ[[q]]$-linear map $P_q\colon \cA\to \cA$ is defined by
\begin{align}
\label{eq:opP}
	P_q[f](t):=\sum_{n\geq 0}E_q^n[f](t). 
\end{align}
Furthermore, the \emph{$q$-difference operator} $D_q\colon \cA \to \cA$ is defined as $D_q:=\Id - E_q.$ We have the following known result:

\begin{proposition}[\cite{Castillo13b}]
\label{prop:MRBO}~
\begin{enumerate}[(i)]
  \item The pair $(\cA,P_q)$ is a RBA of weight $\lambda=-1$.
  \item For any $f,g\in \cP_{\geq 1}$ the operator $D_q$ satisfies the generalized Leibniz rule, i.e.,
  	\begin{align*}
   		D_q[fg] = D_q[f]g + fD_q[g] - D_q[f]D_q[g].
   	\end{align*}
  \item The operators $P_q$ and $D_q$ are mutually inverse, i.e., $D_q\circ P_q = P_q \circ D_q = \Id$.
\end{enumerate}
\end{proposition}

\begin{remark}{\rm{
Recall that the Jackson integral
$$    	
	\mathcal{J}[f](x) := \int_{0}^{x} f(y) d_qy =(1-q)\:\sum_{n \ge 0} f(q^nx) q^nx  	
$$
is the $q$-analogue of the classical indefinite Riemann integral $R$. For functions $\frac{f(x)}{x}$ -- where the Jackson integral is well defined -- it reduces to 
$$
	(1-q)\:\sum_{n \ge 0} f(q^nx) = \int_{0}^{x} \frac{f(y)}{y} d_qy = (1-q) P_q[f](x),
$$ 
which is the $q$-analogue of the integral operator $J$. Correspondingly, the $q$-analogue of the Euler derivation $\delta$ reduces to $(\Id - E_q)$. 
}}
\end{remark}

\begin{lemma}[\cite{Castillo13b}]\label{lem:characterq}
Let $k_1,\ldots,k_n\in \bZ$. Then we have
\begin{align*}
	\ofz_q(k_1,\ldots,k_n)= P_q^{k_1}[yP_q^{k_2}[y\cdots P_q^{k_n}[y]\cdots]](q). 
\end{align*}
\end{lemma}

Surprisingly enough, the algebraic formalism for $q$MZVs is simpler than in the classical case. Let $X_{-1}:=\{p,d,y\}$. By $W_{-1}$ we denote the set of words on the alphabet $X_{-1}$, subject to the rule $pd=dp=\be$, where $\be$ denotes the empty word. Again, $\cA_{-1}$ denotes the algebra spanned by the words in $W_{-1}$, i.e., $\cA_{-1}:=\langle W_{-1} \rangle_\bQ$.
Then we define the product $\shm \colon \cA_{-1}\otimes \cA_{-1} \to \cA_{-1}$ by $\be\shm  w  := w \shm  \be := w$ for any $w\in W_{-1}$, and for any words $u,v\in W_{-1}$
\allowdisplaybreaks{
\begin{enumerate}[(i)]
 \item $yu\shm  v := u \shm  yv := y(u\shm v)$,
 \item $pu \shm  pv := p(u\shm  pv) + p(pu\shm  v) -p(u \shm v)$,
 \item $du \shm  dv := u  \shm dv +du \shm  v -d(u \shm  v)$,
 \item $du\shm  pv = pv\shm  du := d(u\shm  pv)+du \shm  v - u\shm v$.
\end{enumerate}}

\begin{remark}{\rm{
We can deduce (iv) from (iii).}} 
\end{remark}
\begin{lemma}\label{lem:qalgebra}
The pair $(\cA_{-1},\shm )$ is a commutative, associative and unital algebra. 
\end{lemma}

\begin{proof}
The proof is similar to that of \cite[Theorem~7]{Castillo13b}, and left to the reader. 
\end{proof}

\noindent Next we introduce the set  of words ending in the letter $y$ and containing the empty word
\begin{align*}
	Y_{-1}:=W_{-1}y \cup\{\be\}\subseteq W_{-1},
\end{align*}
subject to the rule $pd=dp=\be$. Note that $(\langle Y_{-1} \rangle_\bQ,\shm )$ is a subalgebra of $(\cA_{-1},\shm )$. Moreover we introduce the map $\ofz_q^\shuffle\colon \langle Y_{-1} \rangle_\bQ \to \bQ[[q]]$ by 
\begin{align*}
	p^{k_1}y\cdots p^{k_n}y \mapsto \ofz_q^\shuffle(p^{k_1}y\cdots p^{k_n}y) := \ofz_q(k_1,\ldots,k_n)
\end{align*}
for any integers $k_1,\ldots,k_n$.

\begin{lemma}[\cite{Castillo13b}]
The map $\ofz_q^\shuffle$ is an algebra morphism. 
\end{lemma}


\subsection{General Word Algebraic Part}
\label{ssect:genalg}

In this section we explore the algebraic structure that is related to non-positive arguments for MZVs and  $q$MZVs. For this reason we introduce a parameter $\lambda \in \bQ$. The case $\lambda=0$ corresponds to MZVs and the case $\lambda=-1$ to (modified) $q$MZVs.\\
 
Let $L:=\{d,y\}$ be an alphabet of two letters. The free monoid of $L$ with empty word $\be$ is denoted by $L^\ast$. We denote the free algebra of $L$ by $\bQ\langle L\rangle$ and define the subspace of words ending in $d$ by 
\begin{align*}
	\cT_-:=\cT\cap \bQ\langle L\rangle=\langle \left\{wd\colon  w\in L^\ast \right\}  \rangle_\bQ\subseteq \bQ\langle L\rangle, 
\end{align*}
with $\cT$ defined in Lemma \ref{lem:idealT}. 
The set of \emph{admissible words} is defined as 
\begin{align*}
	Y:=L^\ast y\cup \{\be\},
\end{align*}
and the $\bQ$-vector space spanned by $Y$ is notated as $\cH:=\langle Y \rangle_\bQ$. It is isomorphic to the quotient $\bQ\langle L \rangle / \cT_-.$ The \emph{weight} $\wt(w)$ of a word $w \in Y$ is given by the number of letters of $w$, and we use the convention $\wt(\be):=0$.
Furthermore, the \emph{depth} $\dpt(w)$ of a word $w \in Y$ is given by the number of $y$ in $w$. The $\bQ$-vector space $\cH$ is graded by depth, i.e.,
\begin{align*}
\cH = \bigoplus_{n\geq 0} \cH_{(n)}
\end{align*}
with $\cH_{(n)}:=\langle w\in Y\colon \dpt(w)=n \rangle_\bQ$.


\subsubsection{The algebra $\cH_\lambda,\,\lambda \neq 0$}
\label{ssect:wordbialgebra}

Let $\lambda \in \bQ\setminus \{0\}$. We define the product
\begin{align*}
	\shl \colon \bQ\langle L \rangle \otimes \bQ\langle L \rangle \to \bQ\langle L \rangle
\end{align*}
iteratively by 
\allowdisplaybreaks{ 
\begin{enumerate}[(P1)]
 \item $\be \shl w := w \shl \be := w$ for any $w\in L^\ast$;
 \item $yu\shl v := u \shl yv := y(u\shl v)$ for any $u,v\in L^\ast$;
 \item $du\shl dv:=\frac 1\lambda \big[d(u\shl v)-du\shl v-u \shl dv\big]$ for any $u,v\in L^\ast$.
\end{enumerate}}


\noindent Furthermore, we define the unit map $\eta \colon \bQ \to \bQ\langle L \rangle$, $1 \mapsto \be$. 

\begin{proposition}
\label{prop:algebraH}
For $\lambda \in \bQ$, the triple $(\bQ\langle L \rangle, \shl,\eta)$ is a commutative, associative, and unital $\bQ$-algebra. The subspace $\cT_-$ is a two-sided ideal of $\bQ\langle L\rangle$.   
\end{proposition}

\begin{proof}
In the case $\lambda=-1$ the proof is a consequence of Lemma \ref{lem:qalgebra}. We give a proof for any $\lambda\neq 0$ for completeness, although it could be derived from the case $\lambda=-1$ by appropriate rescaling. Commutativity is clear from the definition. We only have to verify associativity if all words begin with a letter $d$. We apply induction on the sum of the lengths of the words. The base case is trivial. For the inductive step we observe for $a,b,c \in L^\ast$ that 
\allowdisplaybreaks{
 \begin{align*}
  (da\shl db)\shl dc 
  =&~\frac{1}{\lambda} \left[d(a\shl b)-da\shl b - a\shl db \right]\shl dc\\
  =&~\frac{1}{\lambda^2} \left[ d((a\shl b)\shl c) - d(a\shl b)\shl c - (a\shl b)\shl dc \right] \\
  & - \frac{1}{\lambda} \left[ (da\shl b)\shl dc + (a\shl db)\shl dc\right] \\
  =&~\frac{1}{\lambda^2} \left[d(a\shl b\shl c)-a\shl b\shl dc -da\shl b\shl c - a\shl db \shl c \right]\\
  & -\frac{1}{\lambda} \left[ da\shl db\shl c + da\shl b\shl dc + a\shl d b \shl dc \right] \\
  \end{align*}}
  and 
  \allowdisplaybreaks{
  \begin{align*}
  da\shl (db\shl dc)
  =&~\frac{1}{\lambda} da \shl \left[ d(b\shl c) -db\shl c-b\shl dc \right]  \\
  =&~\frac{1}{\lambda^2} \left[ d(a\shl(b\shl c)) - da\shl(b\shl c) - a\shl d(b\shl c) \right] \\
  & - \frac{1}{\lambda}\left[ da\shl(db\shl c) + da\shl(b\shl dc) \right] \\
  =&~\frac{1}{\lambda^2} \left[ d(a\shl b\shl c)-da\shl b\shl c - a\shl db\shl c - a\shl b\shl dc \right] \\
  &-\frac{1}{\lambda}\left[ a\shl db\shl dc + da\shl db\shl c + da\shl b\shl dc \right], 
 \end{align*}}
which shows associativity. From definition we obtain that $\shl (\cT_- \otimes \bQ\langle L \rangle) = \shl (\bQ\langle L \rangle\ \otimes \cT_-) \subseteq \cT_-$ and therefore $\cT_-$ is a two-sided ideal of $\bQ\langle L \rangle$. 
\end{proof} 


\subsubsection{The coproduct $\overline\Delta_\lambda,\,\lambda\in\bQ$}

\noindent Now we define the coproduct
\begin{align*}
 \odl \colon \bQ\langle L \rangle \to \bQ\langle L \rangle \otimes \bQ\langle L \rangle
\end{align*}
by 
\allowdisplaybreaks{
\begin{enumerate}[(C1)]
 \item $\odl(y):= \be\otimes y + y\otimes \be$, 
 \item $\odl(d):= \be\otimes d + d\otimes \be+ \lambda d\otimes d$, 
\end{enumerate}}
\noindent which extends uniquely to an algebra morphism (with respect to~concatenation) on the free algebra $\bQ\langle L \rangle$. The counit map $\varepsilon\colon \bQ\langle L\rangle \to \bQ$ is given by $\varepsilon(\be) = 1$ and $\varepsilon(w)=0$ for any word $w\in L^*\setminus \{\be\}$. 
\begin{example}\label{ex:ocoproduct} We have
\allowdisplaybreaks{
\begin{align*}
  	\overline\Delta_\lambda(dy) = &~ \be \otimes dy + dy \otimes \be +d\otimes y +y\otimes d + \lambda dy\otimes d + \lambda d \otimes dy. 
 \end{align*}}
\end{example}
\begin{proposition}\label{prop:coalgebraH}
For $\lambda \in \bQ$ the triple $(\bQ\langle L \rangle,\odl,\varepsilon)$ is a cocommmutative and counital coalgebra, and $\cT_-$ is a coideal of $\bQ\langle L \rangle$. 
\end{proposition}

\begin{proof}
Cocommutativity is clear by definitions (C1) and (C2). The counit axiom is not hard to verify. Finally we have to check coassociativity. We have
\allowdisplaybreaks{
\begin{align*}
 \lefteqn{(\Id \otimes \odl) \odl(d) 
 = (\Id \otimes \odl) (\be \otimes d + d\otimes \be +\lambda d\otimes d)} \\
 = &~\be \otimes \be \otimes d  + \be \otimes d \otimes \be + \lambda \be \otimes d \otimes d + d\otimes \be \otimes \be 
 +\lambda d \otimes \be \otimes d + \lambda d \otimes d \otimes \be + \lambda^2 d \otimes d \otimes d
\end{align*}}
and 
\allowdisplaybreaks{
\begin{align*}
 \lefteqn{ (\odl \otimes \Id) \odl(d) 
 = (\odl \otimes \Id) (\be \otimes d + d\otimes \be +\lambda d\otimes d)} \\
 = &~\be \otimes \be \otimes d  + \be \otimes d \otimes \be + d \otimes \be \otimes \be +\lambda d\otimes d \otimes \be
   +\lambda \be \otimes d \otimes d + \lambda d \otimes \be \otimes d + \lambda^2 d \otimes d \otimes d. 
\end{align*}}
The case $ (\Id \otimes \odl) \odl(y) =(\odl \otimes \Id) \odl(y)$ is easy to see. We immediately obtain 
\begin{align*}
 \odl(\cT_-) \subseteq \cT_- \otimes \bQ\langle L \rangle + \bQ\langle L \rangle \otimes \cT_-, 
\end{align*}
which concludes the proof. 
\end{proof}


\subsubsection{Compatibility properties of the coproduct ($\lambda\neq 0$ case)}

\begin{lemma}
\label{lem:proco}
 For words $u,v\in L^\ast$ we have
 \begin{align}\label{eq:deltay}
  \odl(y)[\odl(u)\shl\odl(v)]= \odl(yu) \shl \odl(v) = \odl(u) \shl \odl(yv)
 \end{align}
 and 
 \allowdisplaybreaks{
 \begin{align}
 \label{eq:deltad}
 \begin{split}
    \odl(d)[\odl(u)\shl \odl(v)] =~
  & \odl(du) \shl \odl(v) + \odl(u)\shl\odl(dv)\\
  &+\lambda [\odl(du)\shl \odl(dv)]. 
 \end{split}
  \end{align}}
\end{lemma}
 
\begin{proof}
Using (P2) and Sweedler's notation, $ \odl(u)=\sum_{(u)} u_1 \otimes u_2$, we obtain for the first equality of \eqref{eq:deltay} 
\allowdisplaybreaks{
 \begin{align*}
  \odl(y)\left[\odl(u)\shl\odl(v)\right] 
  & =  \odl(y)\left[\sum_{(u),(v)}(u_1\shl v_1)\otimes (u_2 \shl v_2) \right] \\
  & = \sum_{(u),(v)}\left[(u_1\shl v_1)\otimes y(u_2 \shl v_2)+y(u_1\shl v_1)\otimes (u_2 \shl v_2) \right] \\
  & = \sum_{(u),(v)}\left[(u_1\shl v_1)\otimes (yu_2 \shl v_2)+(yu_1\shl v_1)\otimes (u_2 \shl v_2) \right] \\
  & = \odl(yu)\shl \odl(v).
 \end{align*}}
 The second equality follows completly analogously. For \eqref{eq:deltad} we observe
 {\allowdisplaybreaks
 \begin{align*}
    & \odl(d)[\odl(u)\shl \odl(v)] \\
  =~& \odl(d)\left[ \sum_{(u),(v)}(u_1\shl v_1)\otimes (u_2 \shl v_2)\right]\\
  =~& \sum_{(u),(v)}\left[ d(u_1\shl v_1)\otimes (u_2 \shl v_2) + (u_1\shl v_1)\otimes d(u_2 \shl v_2)+\lambda d(u_1\shl v_1)\otimes d(u_2 \shl v_2)\right]\\
  =~& \sum_{(u),(v)}\left[ (\be\otimes d + d\otimes \be +\lambda d\otimes d)(u_1\otimes u_2)\shl(v_1\otimes v_2) \right]\\
    &+\sum_{(u),(v)}\left[ (u_1\otimes u_2)\shl(\be\otimes d + d\otimes \be +\lambda d\otimes d)(v_1\otimes v_2) \right]\\
    &+\lambda \sum_{(u),(v)}\left[ (\be\otimes d + d\otimes \be +\lambda d\otimes d)(u_1\otimes u_2)\shl(\be\otimes d + d\otimes \be +\lambda d\otimes d)(v_1\otimes v_2) \right]\\
  =~& \odl(du) \shl \odl(v) + \odl(u)\shl \odl(dv) +\lambda \left[\odl(du)\shl \odl(dv)\right],
 \end{align*}}
which yields the claim.
 \end{proof}


\subsubsection{The Hopf algebra $\cH_\lambda,\,\lambda\neq 0$}

\begin{theorem}\label{theo:HopfQuot}
The  quintuple $\cH_\lambda=(\cH,\shl,\eta,\dl,\varepsilon)$ is a Hopf algebra with 
\begin{align*}
	\dl(w)&:= \odl(w) \mod ( \cT_-\otimes \bQ\langle L \rangle + \bQ\langle L \rangle \otimes \cT_-)
\end{align*}
for any word $w\in Y$, where $\cH=\bQ\langle L \rangle/\cT_-$ is always identified with $\langle Y \rangle_{\bQ}$.
\end{theorem}

\begin{proof}
On the one hand $\cH$ is the quotient by a two-sided ideal $\cT_-$ and therefore $(\cH,\shl,\eta)$ is an algebra. On the other hand $\cT_-$ is a coideal. Hence, $(\cH,\dl,\varepsilon)$ is a coalgebra. Since $\cH$ is connected it suffices to prove that $(\cH,\shl,\eta,\dl,\varepsilon)$ is a bialgebra. 
We show that
$$
	\odl(u'\shl v') = \odl(u')\shl \odl(v') \mod ( \cT_-\otimes \bQ\langle L \rangle + \bQ\langle L \rangle \otimes \cT_-)
$$
by induction on the sum of weights $\wt(u)+\wt(v)$, of the words $u,v\in Y$. The base cases are straightforward. \\
\textbullet~1st case: $u'=yu$ or $v'=yv$. We have with Lemma \ref{lem:proco}
\begin{align*}
	\odl(yu\shl v') 	& = \odl(y(u\shl v')) = \odl(y) \odl(u\shl v') \\
  				& = \odl(y) (\odl(u)\shl \odl(v)) \mod ( \cT_-\otimes \bQ\langle L \rangle + \bQ\langle L \rangle \otimes \cT_-) \\
  				& = \odl(yu)\shl\odl(v) \mod ( \cT_-\otimes \bQ\langle L \rangle + \bQ\langle L \rangle \otimes \cT_-).
\end{align*}
\textbullet~2nd case: $u'=du$ and $v'=dv$. We have with Lemma \ref{lem:proco} and the induction hypothesis
\allowdisplaybreaks{
 \begin{align*}
  & \odl(du\shl dv)  = \frac{1}{\lambda}\left[\odl\left(d(u\shl v)- du\shl v - u \shl dv \right)\right] \\
  & = \frac{1}{\lambda}\left[\odl(d)(\odl(u)\shl \odl(v)) -\odl(du)\shl \odl(v)\right. \\
  & \left.-\odl(u)\shl \odl(dv) \right] \mod ( \cT_-\otimes \bQ\langle L \rangle + \bQ\langle L \rangle \otimes \cT_-)\\
  & = \frac{1}{\lambda}\left[\odl(du)\shl \odl(v) + \odl(u)\shl\odl(dv) + \lambda (\odl(du)\shl \odl(dv))\right. \\
   &\left. -\odl(du)\shl \odl(v)-\odl(u)\shl \odl(dv) \right] \mod ( \cT_-\otimes \bQ\langle L \rangle + \bQ\langle L \rangle \otimes \cT_-)\\
  & = \odl(du)\shl \odl(dv) \mod ( \cT_-\otimes \bQ\langle L \rangle + \bQ\langle L \rangle \otimes \cT_-),
 \end{align*}}
 \ignore{
$\diamond$ For $\lambda=0$ let $u'=d^nu$ and $v'=d^mv$ with $u,v\in Y$ beginning in $y$. Then we have with Lemma \ref{lem:productexp} and  Lemma \ref{lem:proco}
\allowdisplaybreaks{
 \begin{align*}
   & \odl(d^n u\sho  d^m v) \\
 =~& \frac{n}{n+m}\left\{ \sum_{l=n}^{n+m-1}(-1)^{n+l}\od_0(d)\od_0(d^lu\sho  d^{n+m-l-1}v) + (-1)^m \od_0(d^{n+m}u\sho  v ) \right\}\\
   &+ \frac{m}{n+m}\left\{ \sum_{l=m}^{n+m-1}(-1)^{m+l}\od_0(d)\od_0(d^{n+m-l-1}u\sho  d^l v) + (-1)^n \od_0(u\sho  d^{n+m}v ) \right\} \\
 =~& \frac{n}{n+m}\left\{ \sum_{l=n}^{n+m-1}(-1)^{n+l}\od_0(d)(\od_0(d^lu)\sho  \od_0(d^{n+m-l-1}v)) + (-1)^m \od_0(d^{n+m}u\sho  v ) \right\}\\
   &+ \frac{m}{n+m}\left\{ \sum_{l=m}^{n+m-1}(-1)^{m+l}\od_0(d)(\od_0(d^{n+m-l-1}u)\sho  \od_0(d^l v)) + (-1)^n \od_0(u\sho  d^{n+m}v ) \right\} \\ 
 =~& \frac{n}{n+m}\left\{ \sum_{l=n}^{n+m-1}(-1)^{n+l}(\od_0(d^{l+1}u)\sho  \od_0(d^{n+m-l-1}v) + \od_0(d^lu)\sho  \od_0(d^{n+m-l}v))\right.\\
   & \left.+ (-1)^m \od_0(d^{n+m}u\sho  v ) \right\}\\
   &+ \frac{m}{n+m}\left\{ \sum_{l=m}^{n+m-1}(-1)^{m+l}(\od_0(d^{n+m-l}u)\sho  \od_0(d^l v)+\od_0(d^{n+m-l-1}u)\sho  \od_0(d^{l+1} v))\right.\\
   &+ \left.(-1)^n \od_0(u\sho  d^{n+m}v ) \right\} \\ 
 =~& \od_0(d^n u)\sho  \od_0(d^m v)\mod ( \cT_-\otimes \bQ\langle L \rangle + \bQ\langle L \rangle \otimes \cT_-),  
 \end{align*}}
 }
which concludes the proof. 
\end{proof}

 \begin{example}\label{ex:coproduct} For $n\in \bN$ we have $\Delta_\lambda(y^n)  = \sum_{l=0}^n \binom{n}{l}y^l\otimes y^{n-l};$, and
\allowdisplaybreaks{
\begin{align*}
	 \Delta_\lambda(d^ny) = &~ \be \otimes d^ny + d^ny \otimes \be; \\
  	\Delta_\lambda(yd^ny)= &~ \be \otimes yd^ny + y\otimes d^ny + d^ny\otimes y + yd^ny \otimes \be; \\
  	\Delta_\lambda(dyd^ny) = &~  \be \otimes dyd^ny + y\otimes d^{n+1}y +  d^ny\otimes dy + dy\otimes d^ny + d^{n+1}y\otimes y+ dyd^ny \otimes \be \\
  			& + \lambda dy\otimes d^{n+1}y +\lambda d^{n+1}y \otimes dy.\\
 \end{align*}}
\end{example}


\subsubsection{Compatibility between the product and the coproduct ($\lambda= 0$ case)}

Let us now focus on the case $\lambda=0$. Recall from Paragraph \ref{ssect:RBA} that $\cL$ is the two-sided ideal of the (noncommutative and nonassociative) algebra $(\cA_0,\sho)$ generated by the elements
\begin{align*}
 j^k\big(d(u\sho v)-du\sho v-u\sho dv\big),\,k\in\bZ,\, u,v\in W_0y.
\end{align*}
Now let $\cL_-$ be the two-sided ideal of the (noncommutative and nonassociative) subalgebra $(\bQ\langle L \rangle,\sho)$ generated by the elements
\begin{align*}
 d^k\big(d(u\sho v)-du\sho v-u\sho dv\big),\,k\in\bN_0,\, u,v\in L^\ast.
\end{align*}
Further let {$\cL_-^{(2)}:=\cL_-\otimes\bQ\langle L\rangle+\bQ\langle L\rangle\otimes\cL_-$. 

\begin{proposition}\label{prop:comp-delta-zero}
For any $u',v'\in L^\ast$ we have:
\begin{equation}
\overline{\Delta}_0(u'\sho v')=\overline{\Delta}_0(u')\sho \overline{\Delta}_0(v')\mod\cL_-^{(2)}.
\end{equation}
\end{proposition}
\begin{proof}
We use induction on $r:=|u'|+|v'|$. The cases $r=0$ and $r=1$ being immediate. The case $u'=yu$ or $v'=yv$ is easy and left to the reader. In the case $u'=du$ and $v'=dv$ we compute:
\begin{align*}
& \overline{\Delta}_0(du\sho dv) =\overline{\Delta}_0\big(d(u\sho dv)-u\sho d^2v)\big)\\
=&~\overline{\Delta}_0(d)\overline{\Delta}_0(u\sho dv)-\overline{\Delta}_0(u\sho d^2v)\\
=&~\overline{\Delta}_0(d)\big(\overline{\Delta}_0(u)\sho \overline{\Delta}_0(dv)\big)-\overline{\Delta}_0(u\sho d^2v)\mod \cL_-^{(2)}\\
=&~\overline{\Delta}_0(d)\overline{\Delta}_0(u)\sho \overline{\Delta}_0(dv)+\overline{\Delta}_0(u)\sho \overline{\Delta}_0(d)\overline{\Delta}_0(d) \overline{\Delta}_0(v)-\overline{\Delta}_0(u\sho d^2v)\mod \cL_-^{(2)},
\end{align*}
hence we get:
\begin{equation}
\overline{\Delta}_0(du\sho dv)-\overline{\Delta}_0(du)\sho\overline{\Delta}_0(dv)=-\big(\overline{\Delta}_0(u\sho d^2v)-\overline{\Delta}_0(u)\sho \overline{\Delta}_0(d^2v)\big)\mod \cL_-^{(2)}.
\end{equation}
Iterating this process we return to the case when one of the arguments starts with a $y$.
\end{proof}


\subsubsection{The Hopf algebra $\cH_0$}

Let $\cH_0:=\bQ\langle L\rangle /(\cL_-+\cT_-)$. 

\begin{proposition}\label{prop:coideal-L}
The ideal $\cL_-$ is a coideal of $(\bQ\langle L\rangle,\overline{\Delta}_0)$, where $\overline{\Delta}_0$ is defined by (C1) and (C2) with $\lambda =0$.
\end{proposition}
\begin{proof}
Using Proposition \ref{prop:comp-delta-zero} we compute
\begin{align*}
 &\overline\Delta_0 \big(d(u\sho v)-du\sho v-u\sho dv\big)\\
=&~\overline\Delta_0(d)\big(\overline\Delta_0(u)\sho \overline\Delta_0(v)\big)-\overline\Delta_0(du)\sho\overline\Delta_0(v)-\overline\Delta_0(u)\sho \overline\Delta_0(dv)\mod\cL_-^{(2)}\\
=&~\overline\Delta_0(d)\big(\overline\Delta_0(u)\sho \overline\Delta_0(v)\big)-\overline\Delta_0(d)\overline\Delta_0(u)\sho\overline\Delta_0(v)-\overline\Delta_0(u)\sho \overline\Delta_0(d)\overline\Delta_0(v)\mod\cL_-^{(2)}.
\end{align*}
Hence,  $\overline\Delta_0\big(d(u\sho v)-du\sho v-u\sho dv\big)\in\cL_-^{(2)}$.
\end{proof}
\begin{corollary}
$\cH_0$ is a commutative Hopf algebra.
\end{corollary}
\begin{proof}
From Proposition \ref{prop:coalgebraH} and Proposition \ref{prop:coideal-L} we get that the ideal $\cT_-+\cL_-$ is also a coideal of $\bQ\langle L\rangle$. Hence the quotient $\cH_0$ is a bialgebra. It is graded by the depth (and weight), hence connected, thus $\cH_0$ is a Hopf algebra.
\end{proof}
We will denote by $\Delta_0$ the coproduct on $\cH_0$. By a slight abuse of notation the coproduct on $\bQ\langle L\rangle/\cT$, which was deduced from $\overline\Delta_0$, will also be denoted $\Delta_0$.  Note that the ideal $\cT_-+\cL_-$ is also graded by depth. One then gets a grading on the quotient $\cH_0$, which we still denote by $\dpt$.


\subsubsection{Shuffle factorization}

\noindent Let $\lambda\in\bQ$, including the case $\lambda=0$. From connectedness we can always write 
 \begin{align*}
 	\dl([w]) = \be \otimes [w] + [w] \otimes \be + \tdl([w])\hspace{0.5cm} 
	\text{with}\hspace{0.5cm} \tdl([w])\in \bigoplus_{p+q=n \atop p\neq 0,q\neq 0}\cH_{(p)}\otimes \cH_{(q)}. 
\end{align*}
Therefore in the following we use two variants of Sweedler's notation 
\begin{align*}
	\dl([w])=\sum_{([w])}[w]_1\otimes [w]_2\hspace{0.5cm}
	\text{and}\hspace{0.5cm} 
	\tdl([w])=\sum_{([w])}[w]'\otimes [w]''.
\end{align*}

\noindent The following theorem, valid for any $\lambda$, including $\lambda=0$, provides a nice example of the theory outlined in \cite{Patras93}.

\begin{theorem}\label{theo:trivialrenorm}
 Let $\lambda \in \bQ$. Then for all $w\in L^*$ we have 
 \begin{align*}
  \shl  \circ \dl ([w]) = 2^{\dpt([w])}[w],
 \end{align*}
 where $[w]$ stands for the class of $w$ modulo $\cT_-$ in the case $\lambda\neq 0$ (resp. modulo $\cT_-+\cL_-$ in the case $\lambda=0$).
\end{theorem}

\begin{proof}
We prove this by induction on the weight of $[w]$. For $\wt([w])=0$ we have $[w]=\be$ and obtain$\shl \circ \dl(\be) = \shl(\be \otimes \be) = \be.$ For the inductive step we consider two cases: \\
 \noindent\textbullet~1st case: $w=yv$ with $v\in L^*$\\
 We have $\dpt([w]) = \dpt([v])+1$ and obtain
 \allowdisplaybreaks{
 \begin{align*}
  \shl \circ \dl([yv]) 
  & = \shl \circ (\dl([y])\dl([v]))  = \shl \left(\sum_{([v])}[v]_1\otimes [yv]_2 +\sum_{([v])}[yv]_1\otimes [v]_2\right) \\
  & =  [y] \left(\sum_{([v])}\shl([v]_1\otimes [v]_2) +\sum_{([v])}\shl([v]_1\otimes [v]_2)\right) = 2[y] (\shl\circ \dl([v])) \\
  & = 2^{\dpt([v])+1}[yv]  = 2^{\dpt([w])}[w].
 \end{align*}}
 
 \noindent\textbullet~2nd case: $w=dv$ with $v\in L^*$\\
 Since $\dpt([w])=\dpt([v])$ we observe
 \allowdisplaybreaks{
 \begin{align*}
  \shl \circ \dl([dv]) 
  & = \shl \circ (\dl([d])\dl([v]))\\
  &= \shl \left(\sum_{([v])}[dv]_1\otimes [v]_2 + [v]_1\otimes [dv]_2 + \lambda ([dv]_1\otimes [dv]_2) \right)\\
  & = [d] \left(  \sum_{([v])} \shl ( [v]_1 \otimes [v]_2) \right) = [d](\shl \circ \dl([v]))\\
  & = 2^{\dpt([v])}[dv]= 2^{\dpt([w])}[w].
 \end{align*}}
\end{proof}

\begin{corollary}
\label{cor:decomposition}
 Let $\lambda \in \bQ$. 
 Then for all $w \in L^*$, we have 
\begin{align*}
  (2^{\dpt([w])}-2) [w] 	&= \shl  \circ \tdl([w])\\
  				&= \sum_{([w])} [w]' \shl [w]'' = K \star K([w]).
\end{align*}
The linear map $K:=\Id - \eta \circ \varepsilon \in {\mathrm{End}}_\bQ(\cH_\lambda)$ is a projector to the augmentation ideal $\cH':=\bigoplus_{n > 0} \cH_n$, and $f \star g := \shl \circ (f \otimes g)  \circ \Delta_\lambda$, $f,g\in {\mathrm{End}}_\bQ(\cH_\lambda)$. 
\end{corollary}


\subsubsection{A combinatorial description of the coproduct $\dl$} 
\label{ssect:combCoprod}

In the following we give a combinatorial description of the coproduct $\dl$. However, note that we consider the construction only on an admissible representative $w \in Y$ of a given equivalence class in $\mathbb{Q}\langle L \rangle/ \cT_-$ for $\lambda\neq 0$ (resp. in $\mathbb{Q}\langle L \rangle/ (\cL_-+\cT_-)$ for $\lambda= 0$).   

Let $w:= d^{n_1-1}y \cdots d^{n_{k-1}-1}yd^{n_k-1}y \in Y$ be a word, with $n:=\sum_{i=1}^kn_i$, and define for $1 \leq m \leq k$,  $N^m_w:=\{n_1,n_1+n_2,\ldots,n_1+ \cdots + n_{m}\}$. The coproduct $\dl(w)$ can be calculated as follows. Let $S:=\{s_1 < \cdots < s_l\} \subseteq [n]:=\{1,\ldots,n\}$ and $\bar S:=[n]\backslash S=\{\bar s_1 < \cdots < \bar s_{n-l}\} $. Define the words $w_S:= w_{s_1} \cdots w_{s_l}$ and $w_{\bar S}:= w_{\bar s_1} \cdots w_{\bar s_{n-l}}$. The set $S$ is called \emph{admissible} if both $w_S$ and $w_{\bar S}$ are in $Y$, i.e., if $s_l,\bar s_{n-l} \in N^k_w$ with $w\in Y$. The coproduct is then given by
 \begin{align}
 \label{lambdaCoproduct}
	\dl(w) = \sum_{S \subseteq [n] \atop S\ {\rm{adm}}} w_S \otimes w_{\bar S} 
	+  \sum_{S \subseteq [n] \atop S\ {\rm{adm}}} \sum_{J=\{j_1 < \cdots < j_p\} \subset S \atop {J \neq \emptyset,\ J \cap N^k_w 
	= \emptyset \atop j_p < n_1+ \cdots + n_{k-1}}} \lambda^{|J|}  w_S \otimes w_{[n] \backslash (S \backslash J)}.
\end{align}    

For $\lambda=0$ this reduces to the coproduct corresponding to MPLs at non-positive integer arguments 
$$
	\Delta_0(w) := \sum_{S \subseteq [n] \atop S\ {\rm{adm}}} w_S \otimes w_{\bar S}. 
$$

We introduce now a graphical notation, which should make the above more transparent. The set of vertices $V:=\{\d ,\y\}$ is used to define a polygon. The black vertex $\d \sim d$ and the white one $\y \sim y$. To each admissible word $w = d^{n_1-1}y \cdots d^{n_{k-1}-1}yd^{n_k-1}y \in Y$ corresponds an polygon with clockwise oriented edges, and the vertices colored clockwise according to the word $w$. For instance, the word $w=ddydy$ corresponds to
$$
	\scalebox{0.7}{\ddydy}
$$    
An admissible subset $S \subseteq [n]$ corresponds to a sub polygon. The admissible subsets for the above example are as follows:
$\{1,3\}, \{2,3\}, \{4,5\}$ correspond respectively to
$$
	\scalebox{0.7}{\PdyA \quad 
	\PdyB \quad  
	\PdyC}
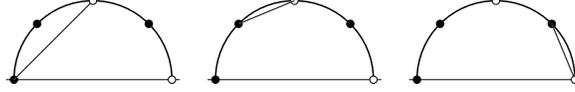
$$
$\{1,2,3\}, \{2,4,5\}$, $\{1,4,5\}$ correspond respectively to
$$
	\scalebox{0.7}{\PddyA \quad 
	\PddyB \quad 
	\PddyC}
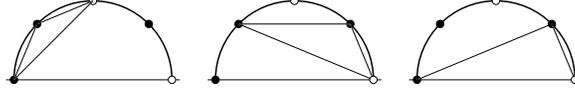
$$
$\{1,2,4,5\}$ and $\{3\}$ correspond respectively to
$$
	\scalebox{0.7}{\Pdddy}
	\quad 
	\scalebox{0.7}{\Py}
$$
In the last polygon we have marked the single vertex subpolygon by a circle around the white vertex.
The coproduct 
\begin{align*}
	\Delta_0(ddydy) 	&=\sum_{S \subseteq [n] \atop S\ {\rm{adm}}} w_S \otimes w_{\bar S} \\
				&= ddydy \otimes \be + \be \otimes ddydy 
				+ 3 dy \otimes ddy + 3 ddy \otimes dy + dddy \otimes y + y \otimes dddy.
\end{align*}
The first two terms on the right-hand side correspond to $S=[n]$ and $S=\emptyset$, respectively.

The pictorial description of the weight-$\lambda$ coproduct (\ref{lambdaCoproduct}) is captured as follows. The second term on the right-hand side of the coproduct (\ref{lambdaCoproduct}) reflects the term $ \lambda d\otimes d$ in the coproduct 
$$
	\odl(d):= \be\otimes d + d\otimes \be+ \lambda d\otimes d
$$
defined further above. It amounts to a certain doubling of those black vertices in an admissible word $w = d^{n_1-1}y \cdots d^{n_{k-1}-1}yd^{n_k-1}y \in Y$, which appear before the $y$ at position $n_1 + \cdots + n_{k-1}$. Algebraically this means that extracting a subpolygon corresponding to the admissible subset $S \subseteq [n]$ leads to a splitting of the word $w$ into $w_S$ and $w_{\bar S '}$, where the augmented complement sets ${\bar S '}$ contain $\bar S$, i.e., $\bar S  \subset {\bar S '}$. This is due to not eliminating several black vertices, i.e., $d$'s that appear before the $y$ at position $n_1 + \cdots + n_{k-1}$. Pictorially we denoted this by doubling black vertices. Returning to the example above. We find for the admissible word $w=ddydy$         
$$
	\scalebox{0.7}{\PdyAx \quad 
	\PdyBx }
$$
which correspond to the admissible subsets and related augmented complement sets: $\{1,3\}$, ${\bar S '}=\{1,2,4,5\}$, and $\{2,3\}$, ${\bar S '}=\{1,2,4,5\}$, respectively. Next we consider
$$
	\scalebox{0.7}{\PddyBx \quad 
	\PddyCx}
$$
where we have $\{2,4,5\}$, ${\bar S '}=\{1,2,3\}$, and $\{1,4,5\}$, ${\bar S '}=\{1,2,3\}$, respectively. For
$$
	\scalebox{0.7}{\PddyAxa \quad 
	\PddyAxb \quad 
	\PddyAxx}
$$
we have $\{1,2,3\}$, ${\bar S '}=\{1,4,5\}$, and $\{1,2,3\}$, ${\bar S '}=\{2,4,5\}$, and $\{1,2,3\}$, ${\bar S '}=\{1,2,4,5\}$, respectively. Finally
$$
	\scalebox{0.7}{\Pdddyxa \quad 
	\Pdddyxb \quad 
	\Pdddyxx}
$$
where we have $\{1,2,4,5\}$, ${\bar S '}=\{1,3\}$, and $\{1,2,4,5\}$, ${\bar S '}=\{2,3\}$, and $\{1,2,4,5\}$, ${\bar S '}=\{1,2,3\}$, respectively.

The coproduct
\allowdisplaybreaks{
\begin{align*}
	\lefteqn{\Delta_\lambda(ddydy) 	=\sum_{S \subseteq [5] \atop S\ {\rm{adm}}} w_S \otimes w_{\bar S} 
	+  \sum_{S \subseteq [5] \atop S\ {\rm{adm}}} \sum_{J=\{j_1 < \cdots < j_p\} \subset S \atop {J \neq \emptyset,\ J \cap N^2_{ddydy} = \emptyset \atop j_p < 3}} \lambda^{|J|}  w_S \otimes w_{[5] \backslash (S \backslash J)}}
\\
				&= ddydy \otimes \be + \be \otimes ddydy 
					+ 3 dy \otimes ddy + 3 ddy \otimes dy + dddy \otimes y + y \otimes dddy\\
				& \quad	+ 2 \lambda dy \otimes dddy 
					+ 4 \lambda ddy \otimes ddy 
					+ 2 \lambda dddy \otimes dy	+ \lambda^2 ddy \otimes dddy	
					+ \lambda^2 dddy \otimes ddy.
\end{align*}}
Again, the first two terms on the right-hand side correspond to $S=[n]$ and $S=\emptyset$, respectively.


\section{Renormalization of regularized MZVs}
\label{sect:Renorm}

Alain Connes and Dirk Kreimer discovered a Hopf algebraic approach to the BPHZ renormalization method in perturbative quantum field theory \cite{Connes00,Connes01}. See \cite{Manchon08} for a review. One of the fundamental results of these seminal works is the formulation of the process of perturbative renormalization in terms of a factorization theorem for regularized Hopf algebra characters. We briefly recall this theorem, and apply it in the context of the Hopf algebra introduced on $t$- and $q$-regularized MPLs, when considering them at non-positive arguments.


\subsection{Connes--Kreimer Renormalization in a Nutshell}
\label{ssect:CKrenorm}

We regard the commutative algebra $\cA:=\bQ[z^{-1},z]]$ with the renormalization scheme 
$\cA = \cA_{-}\oplus \cA_{+},$ where $\cA_{-}:=z^{-1}\bQ[z^{-1}]$ and $\cA_{+}:=\bQ[[z]]$. 

On $\cA$ we define the corresponding projector $\pi: \cA \to \cA_{-}$ by 
\begin{align*}
 \pi\left(\sum_{n=-k}^\infty a_n z^n \right):= \sum_{n=-k}^{-1}a_n z^n
\end{align*}
with the common convention that the sum over the empty set is zero. Then $\pi$ and $\Id-\pi: \cA \to \cA_{+}$ are Rota--Baxter operators of weight $-1$. See e.\,g. \cite{Connes00,Ebrahimi02,Ebrahimi07}.

Let $(\cH,m_{\cH},\Delta)$ be a bialgebra and $(\cA,m_{\cA})$ an algebra. Then we define the convolution product $\star\colon \Hom(\cH,\cA) \otimes \Hom(\cH,\cA) \to \Hom(\cH,\cA)$ by the composition 
\begin{align*}
 \cH \stackrel{\Delta}{\longrightarrow}\cH\otimes \cH \stackrel{\varphi\otimes \psi}{\longrightarrow} \cA\otimes \cA \stackrel{m_\cA}{\longrightarrow}\cA
\end{align*}
for $\varphi,\psi \in \Hom(\cH,\cA)$, or in Sweedler's notation
\begin{align*}
 (\varphi \star \psi) (x):= m_{\cA}(\varphi\otimes \psi)\Delta(x)=\sum_{(x)}\varphi(x_1)\psi(x_2).
\end{align*}

The Connes--Kreimer Hopf algebra approach unveiled a beautiful encoding of one of the key concepts of the renormalization process, i.e., Bogoliubov's counter term recursion, in terms of an algebraic Birkhoff decomposition: 

\begin{theorem}[\cite{Connes00},\cite{Connes01},\cite{Manchon08},\cite{Ebrahimi07a}]
\label{theo:ConKre}
Let $(\cH,m_{\cH},\Delta)$ be a connected filtered Hopf algebra and  $\cA$ a commutative unital algebra equipped with a renormalization scheme $\cA=\cA_{-}\oplus \cA_{+}$ and corresponding idempotent Rota--Baxter operator $\pi$, where $\cA_{-}=\pi(\cA)$ and $\cA_{+}=(\Id-\pi)(\cA)$. Further let $\phi\colon \cH \to \cA$ be a Hopf algebra character. Then: 
\allowdisplaybreaks{
 \begin{enumerate}[a)]
  \item The character $\phi$ admits a unique decomposition 
  \begin{align}
  \label{Birkhoff}
   \phi=\phi_{-}^{\star{(-1)}}\star \phi_{+}
  \end{align}
  called \emph{algebraic Birkhoff decomposition}, in which $\phi_{-}\colon \cH \to \bQ \oplus \cA_{-}$ and $\phi_{+}\colon \cH \to \cA_{+}$ are characters. 
 \item The maps $\phi_{-}$ and $\phi_{+}$ are recursively given fixed point equations
 \begin{align}
  \phi_{-} & = e -\pi\left(\phi_{-} \star (\phi - e) \right),  		\label{Bogo1}\\
  \phi_{+} & = e + (\Id-\pi)\left(\phi_{-}\star (\phi - e\right)),  	\label{Bogo2}
 \end{align}
 where the unit for the convolution algebra product is $e=\eta_{\cA} \circ \varepsilon$, and $\eta_{\cA}:  \bQ \to \cA$ is the unit map of the algebra $\cA$. 
  \end{enumerate}}
\end{theorem}


\subsection{Renormalization of MZVs}
\label{ssect:renormMZV}

An important remark is in order. To improve readability we skip brackets in the notation of classes of words, that is, in the following a word $w$ stands for the class $[w]$.

\smallskip   

Let $k_1,\ldots,k_n\in \bN_0$. Then we define a map $\phi\colon \cH_0 \to \bQ[z^{-1},z]]$  by 
 \begin{align}\label{eq:defphi}
  d^{k_1}y\cdots d^{k_n}y \mapsto \phi(d^{k_1}y\cdots d^{k_n}y)(z) := \partial_z ^{k_1}[x \partial_z ^{k_2}[x \cdots \partial_z^{k_n}[ x ]] \cdots ](z),
 \end{align}
where $x(z):=\frac{e^z}{1-e^z}$.

\begin{lemma}\label{lem:characterphi}
The map $\phi\colon (\cH_0,\sho )\to (\bQ[z^{-1},z]],\cdot)$ is a Hopf algebra character. Furthermore, the following diagram commutes:
   \begin{center}
  \makebox[0pt]{
   \begin{xy}
   (0,20)*+{(\cH_0,\sho )}="a"; (40,20)*+{(\bQ[[t]],\cdot)}="b";
   (40,0)*+{(\bQ[z^{-1},z]],\cdot)}="d";
   {\ar "a";"b"}?*!/_4mm/{\z_t^\shuffle};
   {\ar "b";"d"}?*!/_8mm/{t\mapsto e^z};
   {\ar"a";"d"};?*!/_2mm/{\phi};
   \end{xy}}
 \end{center}
\end{lemma}

\begin{proof}
From the chain and product rule of differentiation we easily obtain that $\phi = \z^{\sh}_{e^z}$. Furthermore, the evaluation maps $t\mapsto e^z$ and $\z_t^\shuffle$ are both algebra morphisms (see Lemma \ref{lem:characterJ}). Therefore $\phi$ is -- as a composition of multiplicative maps -- itself a character.   
\end{proof}

Next we apply Theorem \ref{theo:ConKre} to the character $\phi$ (Lemma \ref{lem:characterphi}). Then we define \emph{renormalized MZVs} $\zeta_+$ -- using the character $\phi_+$ with image in $\bQ[[z]]$ in the Birkhoff decomposition (\ref{Birkhoff}) of $\phi$ -- by 
\begin{align*}
 \zeta_+(-k_1,\ldots,-k_n):=\lim_{z\to 0} \phi_+(d^{k_1}y \cdots d^{k_n}y)(z)
\end{align*}
for $k_1,\ldots,k_n\in \bN_0, n\in \bN$. The first values of $\zeta_+$ in depth two are given in Table \ref{table1} (for an explicit calculation example see Example \ref{ex:reno}). 
Note that $\zeta_+$ respects the shuffle product $\sho $ as $\phi_+$ is a character with respect to the algebra $(\cH_0,\sho )$. However, note that the quasi-shuffle relations are not verified because it would require $\zeta_+(0,0)=\frac{3}{8}$.

\begin{table}
\begin{center}
\renewcommand{\arraystretch}{2}
\begin{tabular}{r|c|c|c|c}
  \backslashbox{$k_1$}{$k_2$}& $0$ & $-1$ & $-2$ & $-3$  \\
  \hline 
  $0$ & $\frac{1}{4}$ & $\frac{1}{24}$ & $0$ & $-\frac{1}{240} $ \\
  \hline
  $-1$ & $\frac{1}{12}$ & $\frac{1}{144}$ & $-\frac{1}{240}$ & $-\frac{1}{1440} $  \\
  \hline
  $-2$ & $ \frac{1}{72} $ & $-\frac{1}{240}$ & $-\frac{1}{720} $ & $\frac{1}{504} $  \\
  \hline
  $-3$ & $-\frac{1}{120}$ & $-\frac{1}{360} $ & $\frac{1}{504}$ & $\frac{107}{100800} $  \\
\end{tabular}
\end{center}
\caption{The renormalized MZVs $\zeta_+(k_1,k_2)$.}\label{table1}
 \end{table}

Next we show that the renormalized MZVs coincide with the meromorphic continuation of MZVs discussed in Section \ref{sect:mero}:

\begin{theorem}\label{theo:mero}
The renormalization procedure is compatible with the meromorphic continuation of MZVs, i.e., for $k\in \bN_0$
 \allowdisplaybreaks{
\begin{align}
\label{eq:mero1}
  \zeta_+(-k) = \z_1(-k)
\end{align}}
and for $a,b\geq 0$ with $a+b$ odd 
\begin{align}\label{eq:mero2}
  \zeta_+(-a,-b) = \z_2(-a,-b). 
\end{align}
\end{theorem}

Note that for $\dpt(w)>2$ there is no information form the meromorphic continuation (see Remark \ref{rem:mero}). 

\begin{proof} 
We begin with the proof of \eqref{eq:mero1}. From Equation \eqref{eq:defphi} we obtain 
 \allowdisplaybreaks{
\begin{align*}
 \phi(d^ky)(z) 
 & = \partial_z^k\left(\frac{e^z}{1-e^z}\right) \\
 & = - \partial_z^k\left(\frac{1}{z}\frac{ze^z}{e^z-1} \right) \\
 & = - \partial_z^k\left( \frac{B_0}{z} + \sum_{n \geq 0}\frac{B_{n+1}}{(n+1)!}z^n \right)\\
 & = -\left(\frac{(-1)^{k}k! B_0}{z^{k+1}} + \sum_{n\geq 0} \frac{B_{n+k+1}}{n+k+1}\frac{1}{n!}z^{n}\right). 
\end{align*}}

Since $d^ky\in Y$ is a primitive element for the coproduct $\Delta_0$ we obtain with Remark \ref{rem:mero} that
\begin{align*}
 \phi_{+}(d^ky)(z) =(\Id-\pi)\phi(z) =  -\frac{B_{k+1}}{k+1} + O(z) = \z_1(-k) + O(z). 
\end{align*}
For Equation \eqref{eq:mero2} we calculate for $a+b$ odd with $a,b\geq 0$ (see Remark \ref{rem:mero})
{\allowdisplaybreaks
\begin{align*}
 \phi(d^ayd^by)(z) 
  =&~\partial_z^a\left[x(z)\partial_z^b\left[x(z)\right] \right] \\
  =&~\partial_z^a\left[\left(\frac{B_0}{z}+\sum_{m\geq 0}\frac{B_{m+1}}{(m+1)!}z^{m}\right) \left( \frac{(-1)^{b}b! B_0}{z^{b+1}} + \sum_{n\geq 0} \frac{B_{n+b+1}}{n+b+1}\frac{1}{n!}z^{n} \right) \right] \\
  =&~\partial_z^a\left[ \text{pole part~}+\sum_{n\geq 0}\frac{B_0B_{n+b+2}}{n+b+2}\frac{z^n}{(n+1)!} + \sum_{m\geq 0}\frac{(-1)^bb!B_0B_{m+b+2}}{(m+b+2)!}z^m \right. \\
   &~\left. +\sum_{l\geq 0} \sum_{\begin{smallmatrix} n+m=l \\ n,m\geq 0 \end{smallmatrix}} \frac{B_{m+1}B_{n+b+1}}{n+b+1}\frac{z^l}{n!(m+1)!} \right] \\
     =&~\text{pole part~} +\frac{B_0B_{a+b+2}}{(a+b+2)(a+1)} + (-1)^b\frac{a!b!B_0B_{a+b+2}}{(a+b+2)!} \\
     &~+\sum_{\begin{smallmatrix} n+m=a \\ n,m\geq 0  \end{smallmatrix}} \frac{B_{m+1}B_{n+b+1}}{(m+1)!(n+b+1)}\frac{a!}{n!} + O(z). 
\end{align*}}
The second and third summand are zero since $a+b+2\geq 3$ is an odd number and therefore $B_{a+b+2}=0$. We have three possibilities for the last sum to be different from zero: 
 \allowdisplaybreaks{
\begin{itemize}
 \item Case 1: $m+1$ and $n+b+1$ are even numbers. Then we have $m+n+b+2=a+b+2$ even, which contradicts that $a+b$ is odd. 
 \item Case 2: $m=0$. The last summand is equal to 
 \begin{align*}
  \frac{B_1B_{a+b+1}}{a+b+1}. 
 \end{align*}
 \item Case 3: $n+b=0$. Then $n=b=0$ and we have for the last summand
 \begin{align*}
  \frac{B_1B_{a+1}}{a+1}. 
 \end{align*}
\end{itemize}}
Therefore we obtain together with Remark \ref{rem:mero} that
 \allowdisplaybreaks{
\begin{align*}
 \phi(d^ayd^by)(z) & = \text{pole part~} +\frac{1}{2}\left(1+\delta_0(b) \right)\frac{B_{a+b+1}}{a+b+1}+ O(z)\\
 & = \text{pole part~} +\z_2(-a,-b)+ O(z). 
\end{align*}}
Let $s:=a+b$ be an odd number and $c,d\geq 0$ with $c+d=s$. Then we observe using the above calculation that
 \allowdisplaybreaks{
\begin{align*}
 \phi_{-}(p^cy)(z)\phi(p^dy)(z) & = \frac{(-1)^{c}c!B_0}{z^{c+1}} \left(  \frac{(-1)^{d}d! B_0}{z^{d+1}} + \sum_{n\geq 0}\frac{B_{n+d+1}}{n+d+1}\frac{1}{n!}z^n  \right)\\
 &= \text{pole part~} +(-1)^c\frac{ B_0B_{c+d+2}}{(c+1)(c+d+2)}+ O(z)
 \end{align*}}
 Since $c+d+2 = a+b+2\geq 3$ is an odd number $B_{c+d+2}=0$ and the constant term is zero. 
Hence, we obtain $\phi_{+}(d^ayd^by)(z) = \z_2(-a,-b) + O(z).$ Here we have used that $\tilde{\Delta}_0$ respects the weight graduation.
\end{proof}

In the light of Theorem \ref{theo:trivialrenorm} and Corollary \ref{cor:decomposition} we deduce a simple way to calculate the renormalized MPL, which is presented in 

\begin{corollary}
\label{cor:primitive-renormalization}
For $w \in Y$, $\dpt(w)>1$
\begin{align}
\label{simple-renormal}
  \phi_{+}(w)= \frac{1}{2^{\dpt(w)}-2} \sum_{(w)}  \phi_{+}(w') \phi_{+}(w'').
\end{align} 
\end{corollary}

Note that both $\dpt(w')$ and $\dpt(w'')$ are strictly smaller than $\dpt(w)$. On the right-hand side of (\ref{simple-renormal}) one can continue to apply Theorem \ref{theo:trivialrenorm} to the words $w',w''$, until $w$ has been fully decomposed into primitive elements. The renormalization of $w$ respectively the corresponding MZV reduce to the simple renormalization of single MPLs at non-positive arguments corresponding to primitive words in $\cH$. For example
$$
	\phi_{+}(dyd^ny) = \phi_{+}(y)\phi_{+}(d^{n+1}y) + \phi_{+}(dy)\phi_{+}(d^{n}y). 
$$

\begin{proof}[Proof of Corollary \ref{cor:primitive-renormalization}]
Statement (\ref{simple-renormal}) follows directly from Theorem \ref{theo:trivialrenorm}, since $\phi_{+}$ is a character by construction. See (\ref{Bogo2}) in Theorem \ref{theo:ConKre}. Observe that (\ref{simple-renormal}) it is compatible with (\ref{Bogo2}), since ${(\Id-\pi)\phi_{+}=\phi_{+}}$.
\end{proof}

\begin{example}\label{ex:reno}{\rm{
Let us calculate the renormalized MZVs $\zeta_+(0,-2)$ and $\zeta_+(-1,-1)$. 
From Example \ref{ex:coproduct} we find
\allowdisplaybreaks{
\begin{align*}
 \tilde{\Delta}_0(yd^2y) = y\otimes d^2y + d^2y \otimes y \hspace{0.5cm}\text{and}\hspace{0.5cm} \tilde{\Delta}_0(dydy) = y\otimes d^2y + d^2y\otimes y + 2dy \otimes dy. 
\end{align*}}
Therefore we obtain form the iterative formulas \eqref{Bogo1} and  \eqref{Bogo2} of Theorem \ref{theo:ConKre}
\allowdisplaybreaks{
\begin{align*}
 \phi_+(yd^2y)  &= (\Id-\pi)\left[\phi(yd^2y) - \left(\phi_{-}(y)\phi(d^2y) + \phi_{-}(d^2y) \phi(y) \right) \right], \\
 \phi_+(dydy) &=(\Id-\pi)\left[ \phi(dydy) - \left(\phi_{-}(y)\phi(d^2y) + \phi_{-}(d^2y) \phi(y) + 2 \phi_{-}(dy)\phi(dy) \right)\right]. 
\end{align*}}
Using 
\allowdisplaybreaks{
   \begin{align*}
  \phi(y)(z) & =-{z}^{-1}-{\frac {1}{2}}-{\frac {1}{12}}z+{\frac {1}{720}}{z}^{3}+  O\!\left( {z}^{4} \right), \\
  \phi(dy)(z)& ={z}^{-2}-{\frac {1}{12}}+{\frac {1}{240}}{z}^{2}+O\!\left( {z}^{4} \right), \\
  \phi(d^2y)(z)& =-2\,{z}^{-3}+{\frac {1}{120}}z-{\frac {1}{1512}}{z}^{3}+ O\!\left( {z}^{4} \right) \\
 \end{align*}}
 and 
 \allowdisplaybreaks{
 \begin{align*}
 \phi(yd^2y)(z)&=  2\,{z}^{-4}+{z}^{-3}+\frac{1}{6}\,{z}^{-2}-{\frac {1}{90}}
-{\frac {1}{240}}\,z+{\frac {1}{30240}}\,{z}^{2}+{\frac {1}{3024}}\,{z
}^{3}+O\left( {z}^{4} \right),  \\
\phi(dydy)(z)&=3\,{z}^{-4}+{z}^{-3}+{\frac {1}{240}}-{\frac {1}{240}}z-{\frac {1}{1008}}{z}^{2}+{\frac {1}{3024}}{z}^{3}+O\!\left( {z}^{4} \right),
\end{align*}}
we observe that
\begin{align*}
 (\Id-\pi)\left[\phi_{-}(y)\phi(d^2y) + \phi_{-}(d^2y) \phi(y)\right](z) &= -\frac{1}{90} + O(z)
\end{align*}
and 
\begin{align*}
 (\Id-\pi)\left[\phi_{-}(y)\phi(d^2y) + \phi_{-}(d^2y) \phi(y)  + 2 \phi_{-}(dy)\phi(dy) \right](z) &= -\frac{1}{360} + O(z).
\end{align*}
Hence,
\begin{align*}
 \phi_+(yd^2y)(z) = O(z) \hspace{1cm} \text{and}\hspace{1cm} \phi_+(dydy)(z) = \frac{1}{240}+\frac{1}{360}+O(z)=\frac{1}{144} + O(z),
\end{align*} 
which results in $\zeta_+(0,-2)=0$ and $\zeta_+(-1,-1)=\frac{1}{144}$. 

Alternatively, we can use the shuffle product to calculate $\zeta_+(0,-2)$ and $\zeta_+(-1,-1)$. Note that since $y\sho  d^2y = yd^2y$ we have $\zeta_+(0,-2)=\zeta_+(0)\zeta_+(-2)=0$. 
Because of $dy\sho  dy = dydy - yd^2y$ we see that $\zeta_+(-1,-1)=\zeta_+(0,-2) + \zeta_+(-1)^2= \frac{1}{144}$.  

A third way to calculate, say, $\zeta_+(-1,-1)$, is based on Corollary \ref{cor:decomposition}, and described in Corollary \ref{cor:primitive-renormalization}: 
\begin{align*}
 \zeta_{+}(-1,-1) = \frac{1}{2}\left( \zeta_+(0)\zeta_+(-2)+ \zeta_+(-2)\zeta_+(0) + 2\zeta_+(-1)\zeta_+(-1) \right) =\frac{1}{144}.
\end{align*}}}
\end{example}


\subsection{Renormalization of $q$MZVs}
\label{ssect:renormqMZV}

The Hopf algebra $(\cH,\shm ,\Delta_{-1})$ is related to the modified $q$-analogue $\ofz_q$ whereas the relation between MZVs and $q$MZVs (see Equation \eqref{eq:lim}) relies on a limit process involving the non-modified $q$MZVs $\fz_q$. Therefore the renormalization related to the $q$MZV deformation is more involved than the renormalization described in the previous section. First of all we apply Theorem \ref{theo:ConKre} in the framework of modified $q$MZVs. We define the map $\psi\colon (\cH,\shm )\to \bQ[z^{-1},z]]$ by 
\begin{align*}
 d^{k_1}y\cdots d^{k_n}y \mapsto \psi(d^{k_1}y\cdots d^{k_n}y)(z):=\sum_{m_1,\ldots,m_n\geq 0}\frac{B_{m_1}}{m_1!}\cdots \frac{B_{m_n}}{m_n!} \cdot C^{k_1,\ldots,k_n}_{m_1,\ldots,m_n} z^{m_1+\cdots+m_n-n}
\end{align*}
for $k_1,\ldots,k_n\in \bN_0$, where 
\begin{align*}
 C^{k_1,\ldots,k_n}_{m_1,\ldots,m_n}:=  \sum_{l_i=0 \atop i=1,\ldots,n}^{k_i}\left(\prod_{i=1}^n\binom{k_i}{l_i}(-1)^{l_i+1}(l_1+\cdots+l_i+1)^{m_i-1} \right). 
\end{align*}

\begin{lemma}\label{lem:qcharacter}
 The map $\psi\colon (\cH,\shm )\to (\bQ[z^{-1},z]],\cdot)$ is a character. Furthermore, the following diagram commutes: 
  \begin{center}
  \makebox[0pt]{
   \begin{xy}
   (0,20)*+{(\cH,\shm )}="a"; (40,20)*+{(\bQ[[q]],\cdot)}="b";
   (40,0)*+{(\bQ[z^{-1},z]],\cdot)}="d";
   {\ar "a";"b"}?*!/_4mm/{\ofz_q^\shuffle};
   {\ar "b";"d"}?*!/_8mm/{q\mapsto e^z};
   {\ar"a";"d"};?*!/_2mm/{\psi};
   \end{xy}}
 \end{center}
\end{lemma}

\begin{proof}
 Let $k_1,\ldots,k_n\in \bN_0$. First we observe that
 \allowdisplaybreaks{
 \begin{align*}
  ~ & \ofz_q^{\sh}(d^{k_1}y\cdots d^{k_n}y) \\
  = & \sum_{m_1>\cdots > m_n >0} q^{m_1}(1-q^{m_1})^{k_1}(1-q^{m_2})^{k_2}\cdots(1-q^{m_n})^{k_n}  \\
  = & \sum_{m_1,\ldots, m_n >0} q^{m_1+\cdots+m_n}(1-q^{m_1+\cdots+ m_n})^{k_1}(1-q^{m_2+\cdots+m_n})^{k_2}\cdots(1-q^{m_n})^{k_n}  \\
  = & \sum_{l_1=0}^{k_1} \cdots \sum_{l_n=0}^{k_n}(-1)^{l_1+\cdots+l_n} \binom{k_1}{l_1} \cdots \binom{k_n}{l_n} \sum_{m_1,\ldots,m_n>0} q^{m_1(l_1+1)}q^{m_2(l_1+l_2+1)} \cdots q^{m_n(l_1+\cdots+l_n+1)} \\
  = & \sum_{l_1=0}^{k_1} \cdots \sum_{l_n=0}^{k_n}(-1)^{l_1+\cdots+l_n+n} \binom{k_1}{l_1} \cdots \binom{k_n}{l_n} \frac{q^{l_1+1}}{q^{l_1+1}-1} \cdots \frac{q^{l_1+\cdots+l_n+1}}{q^{l_1+\cdots+l_n+1}-1}. 
 \end{align*}}
This leads to 
\begin{align*}
   \ofz_q^{\sh}(d^{k_1}y\cdots d^{k_n}y)  \stackrel{q\mapsto e^z}{\longmapsto} & \sum_{l_i=0 \atop i=1,\ldots,n}^{k_i} \left(\prod_{j=1}^{n}(-1)^{l_j+1}  \binom{k_j}{l_j} \frac{e^{z(l_1+\cdots+l_j+1)}}{e^{z(l_1+\cdots+l_j+1)}-1} \right) \\
  = &  \sum_{l_i=0 \atop i=1,\ldots,n}^{k_i} \left(\prod_{j=1}^{n}(-1)^{l_j+1}  \binom{k_j}{l_j} \sum_{m_j\geq 0} \frac{B_{m_j}}{m_j!} (z(l_1+\cdots+l_j+1))^{m_j-1} \right)\\
  = & \sum_{m_1,\ldots, m_n\geq 0} \frac{B_{m_1}}{m_1!} \cdots \frac{B_{m_n}}{m_n!}\cdot C^{k_1,\ldots,k_n}_{m_1,\ldots,m_n} z^{m_1+\cdots +m_n-n}.
\end{align*}
The map $\psi$ is a character since it is a composition of algebra morphisms.
\end{proof}

Next we reverse the modification process applied in Equation \eqref{eq:modify}. Therefore we apply Theorem \ref{theo:ConKre} to the character $\psi$ and define the \emph{renormalized $q$MZVs} $\fz_{+}$ by 
\begin{align}\label{eq:limmod}
 \fz_{+}(-k_1,\ldots,-k_n):=\lim_{z \to 0} \frac{(-1)^{k_1+\cdots+k_n}}{z^{k_1+\cdots+k_n}}\psi_{+}(d^{k_1}y\cdots d^{k_n}y)(z)
\end{align}
for $k_1,\ldots,k_n\in \bN_0$. 

\begin{theorem}
\label{theo:renoqMZV}
 Let $k_1,\ldots,k_n\in \bN_0$. Then $\fz_{+}(-k_1,\ldots,-k_n)$ is well defined, and we have 
 \begin{align*}
  \fz_{+}(-k_1,\ldots,-k_n) = \zeta_+(-k_1,\ldots,-k_n). 
 \end{align*}
Especially, the renormalized $q$MZVs $\fz_{+}$ respect the shuffle product $\sho $. 
\end{theorem}

For the proof of this theorem we need an auxiliary result: 

\begin{lemma}\label{lem:power}
 We have $\psi_{+}(d^ky)(z) = (-1)^{k}\zeta_+(-k)z^k + O(z^{k+1})$ for all $k\in \bN_0$. 
\end{lemma}

\begin{proof}
 Since $d^ky$ is primitive with respect to the coproduct $\Delta_{-1}$, we obtain from Lemma \ref{lem:qcharacter} that
 \begin{align*}
  \psi_{+}(d^ky)(z)	=(\Id-\pi)\psi(d^ky)(z) 
  			= (\Id-\pi)\left(\sum_{m\geq 0} \frac{B_m}{m!}\cdot C_m^k z^{m-1} \right)
			= \sum_{m>0} \frac{B_m}{m!}\cdot C_m^k z^{m-1}
 \end{align*}
 with $C_m^k = \sum_{l=0}^k\binom{k}{l}(-1)^{l+1}(l+1)^{m-1}$. We have 
 \begin{align*}
  C_m^k 	& =  \sum_{l=0}^k\binom{k}{l}(-1)^{l+1}(l+1)^{m-1}  
  		    = \frac{1}{k+1} \sum_{l=0}^{k+1}\binom{k+1}{l}(-1)^l l^m \\
  		& = \left.\frac{1}{k+1} \delta_z^m\left( \sum_{l=0}^{k+1}\binom{k+1}{l} (-z)^l  \right)\right|_{z=1}
   		   = \left.\frac{1}{k+1} \delta_z^m (1-z)^{k+1} \right|_{z=1} 
 \end{align*}
This shows that $C_m^k=0$ for $m=1,\ldots,k$. Furthermore, we observe that $C_{k+1}^k = (-1)^{k+1}\frac{(k+1)!}{k+1}$, which completes the proof.  
\end{proof}

\begin{proof}[Proof of Theorem \ref{theo:renoqMZV}]
 Let $w:=d^{k_1}y\cdots d^{k_n}y$ with $k_1,\ldots,k_n\in \bN_0$. In order to prove that $\fz_{+}$ is well defined we show $\psi_+(w)\in O(z^{\wt(w)-n})$. We split up the coproduct $\Delta_{-1}(w)$ into two parts
\begin{align}
\label{eq:decom}
 \Delta_{-1}(w)=\Delta_0(w) + \left( \Delta_{-1}(w)-\Delta_0(w) \right). 
\end{align}
From Corollary \ref{cor:primitive-renormalization} we deduce that $\Delta_{-1}(w)$ induces a $\bQ$-linear combination of products with $n$ factors of $\psi_+$ in primitive elements of $\cH$. The part of $\psi_+$ corresponding to $\Delta_0(w)$ in the coproduct factorization is homogeneous in weight $\wt(w)$ and the one related to $\Delta_{-1}(w)-\Delta_0(w)$ has weight greater than $\wt(w)$. Therefore Lemma \ref{lem:power} implies that $\psi_+(w)\in O(z^{\wt(w)-n})$. Hence, the limit in \eqref{eq:limmod} exists. On the one hand we can apply Corollary \ref{cor:primitive-renormalization} to $\phi_+(w)$, defined in the previous section, which corresponds to the factorization induced by $\Delta_0(w)$. On the other hand we can do the same with $\psi_+(w)$. However, this factorization is related to $\Delta_{-1}(w)$.  After dividing by $z^{\wt(w)-n}$ and taking the limit $z\to 0$ only the first part in the decomposition \eqref{eq:decom} of $\Delta_{-1}(w)$ makes a contribution in the factorization of $\psi_+(w)$. Using the fact that the leading factor of $\psi_+(d^ky)(z)$ equals $(-1)^k\zeta_+(-k)z^k$ concludes the proof. 
\end{proof}

\begin{example}\label{ex:renoq}
Let us calculate the renormalized $q$MZV $\fz_{+}(-1,-1)$. From Example \ref{ex:coproduct} we obtain
\begin{align*}
 \tilde{\Delta}_0(dydy) = y\otimes d^2y + d^2y\otimes y +2 dy\otimes dy - dy\otimes d^2y - d^2y\otimes dy, 
\end{align*}
which gives 
\begin{align*}
 \psi_+(dydy) =(\Id-\pi)& \left[ \psi(dydy) - \left(\psi_{-}(y)\psi(d^2y) + \psi_{-}(d^2y) \psi(y) + 2 \psi_{-}(dy)\psi(dy) \right.\right.\\
 				  & \left.\left.- \psi_{-}(dy)\psi(d^2y) - \psi_{-}(d^2y)\psi(dy)\right)\right]
\end{align*}
Using
\allowdisplaybreaks{
\begin{align*}
 \psi(y)(z) &= -{z}^{-1}-\frac{1}{2}-\frac{1}{12}\,z+{\frac {1}{720}}\,{z}^{3}-{\frac {1}{30240}}\,{z}^{5}+O\!\left( {z}^{7} \right),\\
 \psi(dy)(z) &=  -\frac{1}{2}\,{z}^{-1}+\frac{1}{12}\,z-{\frac {7}{720}}\,{z}^{3}+{\frac {31}{30240}}\,{z}^{5}+O\!\left( {z}^{7} \right),\\
 \psi(d^2y)(z) & =  -\frac{1}{3} \,{z}^{-1}+{\frac {1}{60}}\,{z}^{3}-{\frac {1}{168}}\,{z}^{5}+O\!\left( {z}^{7} \right) 
\end{align*}}
and 
\allowdisplaybreaks{
\begin{align*}
 \psi(dydy)(z)={\frac {5}{12}}\,{z}^{-2}+\frac{1}{6}\,{z}^{-1}-\frac{1}{36}+{\frac {1}{216}}\,{z}^{2}-{\frac {1}{120}}\,{z}^{3}-{\frac {19}{9072}}\,{z}^{4}+O\!\left( {z}^{5} \right),
\end{align*}}
we observe that
\begin{align*}
 (\Id-\pi) \left[\psi_{-}(y)\psi(d^2y) + \psi_{-}(d^2y) \psi(y) + 2 \psi_{-}(dy)\psi(dy) \right](z) = -\frac{1}{18}-\frac{1}{135}z^2+O(z^3)
\end{align*}
and 
\begin{align*}
 (\Id-\pi) \left[-\psi_{-}(dy)\psi(d^2y) - \psi_{-}(d^2y)\psi(dy)\right](z) = \frac{1}{36} +\frac{11}{2160}z^2 + O(z^3). 
\end{align*}
Therefore we have $\psi_{+}(dydy)(z) = \frac{1}{144} z^2 + O(z^3)$ and consequently $\fz_{+}(-1,-1)=\frac{1}{144}$, which coincides with $\zeta_+(-1,-1)$. 
\end{example}

\begin{remark}
{\rm{The crucial point in the renormalization of $q$MZVs is the fact that $\psi_+(w)(z)\in O(z^{\wt(w)-n})$. We used a corollary of Theorem \ref{theo:trivialrenorm} to prove this. However, in the light of Theorem \ref{theo:ConKre} the previous example shows that this is obtained by non-trivial cancellations.}}
\end{remark}
 
\bibliographystyle{alpha}
\bibliography{library}
\end{document}